\documentclass[10pt,twoside]{article}

\usepackage[english]{babel}

\usepackage{amsfonts}
\usepackage{amssymb}
\usepackage{amsmath}
\usepackage{mathrsfs}       

\newtheorem{theo}{Theorem}[section]
\newtheorem{prop}{Proposition}[section]
\newtheorem{lem}{Lemma}[section]
\newtheorem{cor}{Corollary}[section]
\newtheorem{DEF}{Definition}[section]
\newtheorem{EX}{Example}[section]
\newtheorem{REM}{Remark}[section]
\newenvironment{theorem}{\begin{theo}}{\end{theo}}  
\newenvironment{proposition}{\begin{prop}}{\end{prop}}  
\newenvironment{lemma}{\begin{lem}}{\end{lem}}
\newenvironment{corollary}{\begin{cor}}{\end{cor}}
\newenvironment{definition}{\begin{DEF}\rm}{\end{DEF}} 
\newenvironment{example}{\begin{EX}\rm}{\end{EX}} 
\newenvironment{remark}{\begin{REM}\rm}{\end{REM}}
\newenvironment{proof}{\noindent \texttt{P\,r\,o\,o\,f.}}{\hfill $\square$ \par \ } 

\newcommand{\R}{\mathbb R}

\newcommand{\N}{\mathbb N}
\newcommand{\B}{\mathbb B}

\newcommand{\X}{\mathbb X}
\newcommand{\Y}{\mathbb Y}

\newcommand{\dom}{{\rm dom}\, }
\newcommand{\grph}{{\rm gph}\,}
\newcommand{\lev}{{\rm lev}\, }

\newcommand{\nullv}{\mathbf{0}}

\newcommand{\inte}{{\rm int}\, }

\newcommand{\valf}{{\rm val}_\mathcal{P}}
\newcommand{\locvalf}[1]{{\rm locval}_{\mathcal{P},\, #1}}
\newcommand{\Argmin}{{\rm Argmin}_\mathcal{P}}

\newcommand{\calm}{{\rm clm}\, }
\newcommand{\dist}{{\rm dist} }
\newcommand{\ball}{{\rm B}}
\newcommand{\uball}{{\B}}

\newcommand{\stsl}[1]{|\nabla #1|}
\newcommand{\sostsl}[1]{\overline{|\nabla #1|}{}^>}
\newcommand{\stslx}[1]{|\nabla_x #1|}
\newcommand{\sostslx}[1]{\overline{|\nabla_x #1|}{}^>}
\newcommand{\fsubdif}{\widehat{\partial}}
\newcommand{\fesubdif}{\widehat{\partial_\epsilon}}
\newcommand{\soesubdif}{\fesubdif^>}
\newcommand{\soesubdifsl}[1]{\overline{|\partial #1|}{}^>}
\newcommand{\FrDer}{\widehat{\rm D}}
\newcommand{\StDer}{\overline{\rm D}}
\newcommand{\GDer}{{\rm D}}
\newcommand{\Coder}{\widehat{\rm D}^*}
\newcommand{\Normal}{\widehat{\rm N}}

\newcommand{\disp}{d[f(\bar p,\cdot),F(\bar p,\cdot)]}
\newcommand{\pdisp}{d[f,F]}
\newcommand{\dispnull}{d[\nullv,F(\bar p,\cdot)]}

\title{\bf On Lipschitz semicontinuity properties of variational
systems with application to parametric optimization}

\author{{\sc A. Uderzo}\footnote{Dept. of Statistics, University of
Milano-Bicocca, Via Bicocca degli Arcimboldi, 8 - 20126 Milano, Italy,
e-mail address: {\it amos.uderzo@unimib.it} }
}

\begin{document}

\maketitle

\vskip3cm

\begin{abstract}
In this paper two properties of recognized interest in variational analysis,
known as Lipschitz lower semicontinuity and calmness, are studied
with reference to a general class of variational systems, i.e.
to solution mappings to parameterized generalized equations. 
In consideration of the metric nature of such properties, some related
sufficient conditions
are established, which are expressed via nondegeneracy conditions
on derivative-like objects appropriate for a metric space analysis.
For certain classes of generalized equations in Asplund spaces,
it is shown how such conditions can be formulated by using
the Fr\'echet coderivative of the field and the derivative of the
base. Applications to the stability analysis of parametric constrained
optimization problems are proposed.
\end{abstract} 

\vskip3cm


\vskip.5cm


\vfill


\section*{0\quad Introduction}

When looking at several areas of pure and applied mathematics,
the same dichotomy appears in a variety of
situations: on one hand, for treating a formalized problem as well
as for understanding modelized phenomena, one reduces to solve
various kinds of equations; on the other hand, one experiences 
very often hard difficulties in finding out an explicit solution,
if any, for sufficiently general classes of equations.
Such dichotomy gains more evidence when parameters,
alongside unknowns, enter the equations under examination.
The presence of the former ones nonetheless is essential, in as
much as it allows to describe effects of errors and/or inaccuracies
frequently arising in real-world measurements and
transmissions of data. Moreover
parameters enable to implement perturbation methods of analysis,
that sometimes can afford interesting theoretical insights into the
original problem. It is clear that, in the presence of parameters,
the solution set associated with an equation becomes a (generally)
set-valued mapping, which is expected to be only implicitly defined,
due to the aforementioned computational difficulties. In such case,
even the solvability of an equation comes to depend on parameters.
As a result of such irreducible dichotomy, one is forced to study
indirectly features and behaviour of the solution mapping, whose
analytic form remains hidden, by performing proper inspections of the given
equation data. Roughly speaking, this is the spirit essentially
shared by many implicit function and multifunction theorems
recently established in different areas of analysis (see
\cite{BorZhu05,ChKrYa11,DurStr12,KlaKum02,LeTaYe08,Mord06,Schi07,Uder09}, just
to mention those works cited in the present paper for other purposes).
The case of parameterized generalized equations\footnote{The terminology
``generalized equations", as well as the problem itself they formalize, were
introduced by S.M. Robinson in \cite{Robi79}.} in variational analysis
makes no exception.

Given a set-valued mapping $F:P\times X\longrightarrow 2^Y$
and a functon $f:P\times X\longrightarrow Y$, by {\it parameterized
generalized equation} the following problem is meant:
$$
    \hbox{find $x\in X$ such that}\ f(p,x)\in F(p,x). \leqno (\mathcal{GE}_p)
$$
According to the nowadays vast literature devoted to such subject,
parameterized generalized equations are formalized in several different fashions.
For the purposes of the current research work, the form
resulting from $(\mathcal{GE}_p)$, adopted also in \cite{LevRoc94,
Uder09}, seems to be the appropriate one.
The problem data $f$ and $F$ are sometimes referred to as the
{\it base} and the {\it field} of $(\mathcal{GE}_p)$, respectively. 
In $(\mathcal{GE}_p)$ the variable denoted by $x$ plays the role
of problem unknown (state), whereas $p$ indicates a varying
parameter.
The related solution mapping $G:P\longrightarrow 2^X$, implicitly
defined by $(\mathcal{GE}_p)$, namely the (generally) set-valued
mapping
$$
    G(p)=\{x\in X:\ f(p,x)\in F(p,x)\},
$$
is often referred to as the {\it variational system} associated with
$(\mathcal{GE}_p)$. Notice that in the formulation considered above
both the base and the field do depend on the parameter.

Correspondingly with the specialized forms taken by the base and
the field of $(\mathcal{GE}_p)$, variational systems appear and
show their relevance in a wide variety of contexts from mathematical
programming, variational analysis, equilibrium and control theories.
More precisely, they may formalize parametric constraint systems
or the set of local/global solutions in parametric optimization problems. In
adequately structured settings, variational conditions in the
Robinson's sense\footnote{The term {\it variational condition}
for generalized equations can be traced back to
\cite{RocWet98}. In fact, it already appears in \cite{LevRoc96}.
S.M. Robinson adopted such term, thereby contributing to popularize it.}, that is
abstract problems able to include optimality conditions, variational
inequalities and complementarity systems, are other examples of
variational systems. Furthermore, even problems seemingly
having not an extremal nature, such as the search for fixed
points and equilibria, can be reduced to variational systems.
The reader is referred to \cite{Mord06b} for a rich account on
possible applications of such formalism.

The aim of the present paper, in accordance with the spirit of
implicit multifunction theorems, is to contribute to the study
of certain properties of variational systems associated with
$(\mathcal{GE}_p)$. Within the huge scientific production devoted
to this theme, a major research line seems to concentrate on
the analysis of such properties as Aubin continuity and metric
regularity, the latter being related to the former through the
inverse mapping (see, for instance \cite{AraMor10,AraMor11,Aubi84,
DurStr12,GeMoNa09,LeTaYe08,Mord08,MorNam09,NgTrTh12,Uder09}).
Indeed, there exist many important findings
about them and the achievements accumulated on the subject
during several decades allow to draw now a clear and comprehensive
theoretical picture.
Nonetheless, there are other forms of Lipschitz behaviour,
whose study seems to be justified by not less strong
motivations. Among them, Lipschitz lower semicontinuity and calmness
are here considered.

Lipschitz lower semicontinuity describes in quantitative terms, by
means of a Lipschitzian type metric estimation, the lower semicontinuity
behaviour of a given multifunction at a reference point. As such,
it requires on both the domain and the range space a metric structure.
If applied to variational systems, such property not only ensures
local solvability of the corresponding parameterized generalized
equations, but provides also estimates of distances of their values
from a certain solution, pertaining to a reference value of the
parameter. In the more particular context of parametric constrained
optimization, the validity of Lipschitz lower semicontinuity property
of the feasible region mapping yields a calmness from above behaviour
of the value function (see Section \ref{subsec:Lippromultfun}).
Apart from being interesting in itself, Lipschitz lower semicontinuity
may be exploited to characterize other stability properties of
multifunctions. For instance, its occurrence at all points of the
graph of a set-valued mapping, with fixed related constants, near
a given point is known to be equivalent to the Aubin continuity at that
point (see \cite{KlaKum02}). Again, a multifunction is Lipschitz
lower semicontinuous at a given point and calm at all points of
its graph near such point iff it satisfies the Aubin property ibidem
(see \cite{KlaKum02}). The latter characterization demonstrates
also a possible employment of calmness, thereby introducing the
second property under study.

Calmness is a property that relaxes at the same time the requirements
of Aubin continuity and those of upper Lipschitz semicontinuity.
Considered by several authors in different contexts and under
various names (see \cite{Burk91,Clar76,RocWet98,YeZhu95}),
it turned out to capture a crucial behaviour of multifunctions.
It has been established to be equivalent, via the inverse mapping,
to metric subregularity, a weaker ``one-point" variant of metric
regularity. As such, it was successfully introduced already in
\cite{Ioff79} for establishing optimality conditions. In more
recent times, since the failure of metric regularity has been proved
by B.S. Mordukhovich for major classes of variational systems
(see \cite{GeMoNa09,Mord08}),
calmness is receiving an increasing interest. On the other hand, in the
particular context of constraint systems analysis, it has been
widely employed as a key condition to guarantee the occurence
of highly desirable phenomena: existence of local error bounds
for the relations formalizing the constraints, viability of penalty
function methods, validity of constraint qualifications (see
\cite{HenOut05,KlaKum02}). More precisely, the interplay of
calmness and the Abadie constraint qualification in relation to
Karush-Kuhn-Tucker points for mathematical programming problems
has been well understood since many years. Furthermore, in
\cite{IofOut08} it has been shown how calmness/metric subregularity
can be utilized to carry out a deep analysis of constraint qualifications
for the basic rules of nonsmooth subdifferential calculus.
Again, calmness plays
a primary role in characterizing weak sharpness of local minimizers
in optimization (see \cite{StuWar99}).

Motivated by the aforementioned theoretical references and applicative
instances, the study reported in the present paper deals mainly with
conditions for detecting Lipschitz lower semicontinuity and calmness
of variational systems associated with $(\mathcal{GE}_p)$. Roughly
speaking, in both the cases a basic sufficient condition is obtained according
to the scheme summarized below:
\vskip.5cm
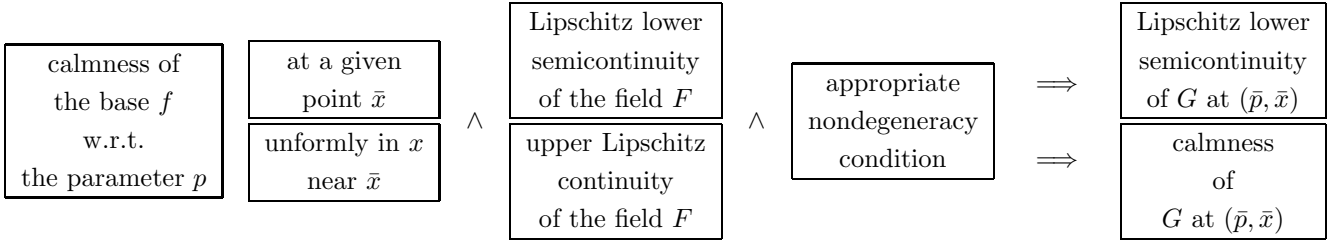
$$
    \begin{tabular}{|c|}
       \hline
      calmness of  \\
       the base $f$ \\
         w.r.t. \\
        the parameter $p$ \\
        \hline
    \end{tabular}
\quad
\begin{tabular}{|c|}
       \hline
       at a given  \\
       point $\bar x$ \\
        \hline \hline
        unformly in $x$ \\
           near $\bar x$ \\
        \hline
    \end{tabular}
\quad
   \wedge
\quad
\begin{tabular}{|c|}
       \hline
      Lipschitz lower  \\
       semicontinuity \\
       of the field $F$ \\
        \hline \hline
        upper Lipschitz \\
        continuity \\
         of the field $F$ \\
        \hline
\end{tabular}
\quad
   \wedge
\quad
\begin{tabular}{|c|}
       \hline
        appropriate  \\
       nondegeneracy \\
        condition \\
        \hline
\end{tabular}
    \quad
\begin{tabular}{c}
       $\Longrightarrow$ \\
                        \\
        $\Longrightarrow$ \\
\end{tabular}
   \quad
\begin{tabular}{|c|}
       \hline
      Lipschitz lower  \\
         semicontinuity \\
       of $G$ at $(\bar p,\bar x)$ \\
        \hline \hline
          calmness \\
         of \\ $G$ at $(\bar p,\bar x)$ \\
        \hline
\end{tabular}
$$
\vskip.5cm
Since the natural environment where to conduct the analysis of
both such properties are metric spaces, the appropriate nondegeneracy
conditions appearing in the above scheme will be expressed in terms of
positivity of certain derivative-like objects, whose relevance
in variational analysis has been exemplarily illustrated in
\cite{FaHeKrOu10}.
Criteria for Lipschitz lower semicontinuity and calmness properties
of set-valued mappings have been studied by many authors (see,
among the others, \cite{IofOut08,KlaKum02}). For special classes
of generalized equations the main efforts have been concentrated
on calmness (see \cite{AraMor10,ZheNg07,ZheNg10}). To the author's
knowledge, a study considering both such properties in the case
of variational systems defined by $(\mathcal{GE}_p)$ has not been
undertaken yet.

The material proposed in the paper is arranged in four sections,
whose contents are briefly outlined below.
In Section \ref{sec:varanapre} needed preliminaries from several
topics of variational analysis are presented. The Lipschitzian type
properties under examination are precisely stated and useful
connections with other stability properties of multifunctions are
recalled. In particular, in view of subsequent applications, some
consequences of such kind of properties on the value function
associated with a family of parametric constrained optimization
problems are discussed. The second part of this section is devoted
to review those tools and constructions of generalized differentiation,
which are involved in the regularity conditions of the main conditions. In Section 
\ref{sec:Liplsccriter} a sufficient condition for the Lipschitz lower semicontinuity
of variational systems is established in a purely metric setting.
In Section \ref{sec:calmcriter} a similar condition for calmness of
variational systems, complemented with an estimation of the
related calmness modulus, is achieved. Such result is then compared
with a quite close very recent calmness condition, which considers
the special case of parameterized generalized equations in Asplund
spaces, having constant (null) base. An analogous condition for
generalized equations with smooth base is also provided.
In Section \ref{sec:applparoptim}
the result formulated in Section \ref{sec:Liplsccriter} is applied
to the study of parametric constrained optimization problems.
More precisely, conditions able to guarantee the Lipschitz lower
semicontinuity of the solution mapping are provided in different
settings.

The notation employed throughout the paper is consistent with
that used in \cite{Mord06,Mord06b}
and in a large part of the variational analysis literature. For the reader's
convenience a list is provided below. Symbol $\R$ stands for
the real number field. $(X,d)$ denotes a metric space.
In such setting, the distance of an element $x$ from a set $S$ is
denoted by $\dist(x,S)$, $\ball (x,r)=\{z\in X: d(x,z)\le r\}$
stands for the closed ball with center at $x$ and radius $r$,
and $\ball (S,r)=\{z\in X:\ \dist(z,S)\le r\}$ for the $r$-enlargement
of $S\subseteq X$, provided that $r\ge 0$.
Given a subset $S$, its indicator function evaluated at $x$ is
denoted by $\iota(x,S)$.  By $\inte S$ the topological interior
of $S$ is denoted. Given a function $\phi:X\longrightarrow
\R\cup\{\pm\infty\}$, $\dom\phi$ denotes the domain of $\phi$
and, given $s\in\R$, $\lev_{{}_{>s}}\phi=\{x\in X:\ \phi(x)>s\}$
the $s$-super level set of $\phi$.
Given  a set-valued mapping (multifunction) $\Phi:X\longrightarrow
2^Y$, $\dom\Phi=\{x\in X:\ \Phi(x)\ne\varnothing\}$ and $\grph\Phi=
\{(x,y)\in X\times Y:\ y\in\Phi(x)\}$ indicate the domain and the graph
of $\Phi$, respectively. Unless otherwise stated, all multifunctions
will be assumed to take closed values.
Whenever $(\X,\|\cdot\|)$ is a Banach space, $\X^*$ will denote
its topological dual, with $\X$ and $\X^*$ being paired in duality by
the bilinear form $\langle\cdot\, ,\cdot\rangle:\X^*\times\X\longrightarrow
\R$. The null element of a vector space is marked by $\nullv$, whereas
the null element of its dual by $\nullv^*$. The unit ball centered
at $\nullv^*$ is denoted by $\B^*$. Whenever $\Lambda:\X
\longrightarrow\Y$ indicates a linear bounded operator between
Banach spaces, $\Lambda^*$ denotes its adjoint operator and
$\|\Lambda\|_\mathcal{L}$ its operator norm. The acronym l.s.c.
will be used for short in place of lower semicontinuous.


\section{Variational analysis preliminaries}    \label{sec:varanapre}


\subsection{Lipschitzian type properties of multifunctions}    \label{subsec:Lippromultfun}

Let us start with recalling the basic properties that are the main
theme of the paper.

\begin{definition}     \label{def:Liplscset}
A set-valued mapping $\Phi:P\longrightarrow 2^X$ between
metric spaces is said to be {\it Lipschitz lower semicontinuous} at
$(\bar p,\bar x)\in\grph\Phi$ if  there exist positive real constants
$\zeta$ and $l$ such that
$$
    \Phi(p)\cap \ball (\bar x,\l\, d(p,\bar p))\ne\varnothing,
    \quad\forall p\in \ball(\bar p,\zeta).
$$
\end{definition}

The following example demonstrates a situation in which, while
lower semicontinuity is in force, Lipschitz lower semicontinuity
fails to hold. It is aimed at clarifying the difference between the
quantitative notion of lower semicontinuity appearing in Definition
\ref{def:Liplscset} and the merely topological one.

\begin{example}
Let $P=X=\R$ be endowed with its usual Euclidean metric structure.
Consider the epigraphical set-valued mapping $\Phi:\R\longrightarrow
2^\R$ associated with the profile $\varphi(p)=\sqrt{|p|}$, that is
$$
    \Phi(p)=\{x\in\R:\ x\ge\sqrt{|p|}\},
$$
and the points $\bar p=\bar x=0$. Mapping $\Phi$ is clearly lower
semicontinuous at $(0,0)$: for any $\epsilon>0$, by taking $\delta_\epsilon
\in (0,\epsilon^2)$ one has
$$
    \Phi(p)\cap\ball(0,\epsilon)\ne\varnothing,\quad\forall p\in
    \ball(0,\delta_\epsilon).
$$
Nonetheless, $\Phi$ is not Lipschitz lower semicontinuous at
$(0,0)$. Indeed, for every $l>0$ and $\zeta>0$ one can find
$p\in\ball(0,\zeta)$ such that $\sqrt{|p|}>l|p|$, and therefore
for such value of $p$ it results in
$$
   \Phi(p)\cap\ball(0,l|p|)=\varnothing.
$$
\end{example}

Another notion related to stability properties of multifunctions is provided by
the next definition.

\begin{definition}    \label{def:calmset}
A set-valued mapping $\Phi:P\longrightarrow 2^X$ between
metric spaces is said to be {\it calm} at $(\bar p,\bar x)\in
\grph\Phi$ if  there exist positive real
constants $\delta$, $\zeta$ and $\ell$ such that
\begin{eqnarray}       \label{inc:calmset}
    \Phi(p)\cap\ball(\bar x,\delta)\subseteq \ball(\Phi(\bar p),\ell\,
     d(p,\bar p)),\quad\forall p\in \ball(\bar p,\zeta),
\end{eqnarray}
or, equivalently,
$$
   \dist(x,\Phi(\bar p))\le\ell d(p,\bar p),\quad\forall
   x\in\Phi(p)\cap \ball(\bar x,\delta),\ 
   \forall p\in \ball(\bar p,\zeta).
$$
Any such constant as $\ell$ is called {\it calmness constant}
for $\Phi$ at $(\bar p,\bar x)$. The value
$$
    \calm\Phi(\bar p,\bar x)=\inf\{\ell>0:\ \exists\delta,\, \zeta>0
     \hbox{ for which (\ref{inc:calmset}) holds}\}
$$
is called {\it calmness modulus} of $\Phi$ at $(\bar p,\bar x)$.
\end{definition}

Calmness has strict connections with other well-known  stability properties
for multifunctions. One one hand, it can be regarded as an ``image
localization" of the notion of upper Lipschitz continuity introduced
by S.M. Robinson (see \cite{Robi79,Robi81}).
Let us recall that a set-valued mapping  $\Phi:P\longrightarrow 2^X$
between metric spaces is called {\it upper Lipschitz} at
$\bar p\in\dom\Phi$ if there exist positive constants
$\zeta$ and $\ell$ such that
\begin{eqnarray}       \label{inc:uLipset}
    \Phi(p)\subseteq \ball(\Phi(\bar p),\ell\,
     d(p,\bar p)),\quad\forall p\in \ball(\bar p,\zeta).
\end{eqnarray}
Clearly, inclusion (\ref{inc:uLipset}) makes (\ref{inc:calmset})
satisfied with any positive $\delta$.
Nonetheless, note that, whereas upper Lipschitz continuity entails
upper semicontinuity of a set-valued mapping, calmness does not.

On the other hand, calmness can be regarded as a ``one-point"
version of the {\it Aubin continuity}, known also under the name of
{\it pseudo-Lipschitz continuity} (or {\it Lipschitz-likeness}),
which postulates the existence of positive reals $\delta$, $\zeta$
and $\ell$ such that
$$
     \Phi(p)\cap\ball(\bar x,\delta)\subseteq \ball(\Phi(p'),\ell\,
     d(p,p')),\quad\forall p,\, p'\in \ball(\bar p,\zeta).
$$
\begin{example}
Let $P=X=\R$ be endowed with its usual metric structure. Consider
the set-valued mapping $\Phi:\R\longrightarrow 2^\R$ defined by
$$
     \Phi(p)=\left\{ \begin{array}{ll}
       [0,+\infty), & \hbox{if } p=0, \\
       (-\infty,-1]\cup\{0\}, & \hbox{otherwise.}
     \end{array}\right.
$$
This mapping is calm at $(\bar p,\bar x)=(0,0)$, because
for any $\ell\ge 0$ it holds
$$
   \Phi(p)\cap\ball(0,\delta)=\{0\}\subseteq\ball([0,+\infty),\ell
    |p|)= [-\ell|p|,+\infty),\quad\forall p\in\R,
$$
provided that $\delta<1$. Notice that $\calm\Phi(0,0)=0$.
Mapping $\Phi$ fails clearly to be upper Lipschitz at $0$.
The reader should notice as well that $\Phi$ fails also to be
Aubin continuous at $(0,0)$, because for every $\ell\ge 0$
and positive $\delta$ and $\zeta$, there is $p\in\ball(0,\zeta)$
such that
$$
     [0,\delta]\not\subseteq  [0,\ell|p|].
$$
\end{example}

\begin{remark}     \label{rem:Liplsccalmgen}
(i) Note that Definition \ref{def:calmset} implies that $\Phi
(\bar p)\ne\varnothing$, while set $\Phi(p)\cap \ball(\bar x,\delta)$
may happen to be empty even for $p$ near $\bar p$.
Definition \ref{def:Liplscset} entails instead the nonemptiness
of the values taken by $\Phi$ near $\bar p$.

(ii) Whenever $\Phi$ is a single-valued mapping, inclusion
(\ref{inc:calmset}) would take the form
$$
    d(\Phi(p),\Phi(\bar p))\le\ell d(p,\bar p),\quad\forall p\in
   \ball(\bar p,\zeta) \hbox{ such that } 
   \Phi(p)\in \ball(\Phi(\bar p),\delta),
$$
so, like in the set-valued case, calmness as resulting from
Definition \ref{def:calmset} does not imply continuity.
Nevertheless, in the most part of variational analysis
literature (see, for instance, \cite{Levy06,RocWet98}), the limitation
$\Phi(p)\in \ball(\Phi(\bar p),\delta)$
is dropped out in the case of functions. It is clear that, with
this variant, calmness for single valued-mappings yields
continuity and, actually, it becomes equivalent to upper
Lipschitz continuity.

(iii) Through simple examples, it is readily shown that Lipschitz
lower semicontinuity and calmness are properties independent 
of one another.
\end{remark}

For scalar functions (i.e. when $X=\R\cup\{\pm
\infty\}$), the notion of calmness has been conveniently splitted
into two distinct ones. 
A function $\phi:P\longrightarrow\R\cup\{\pm\infty\}$
is said to be {\it calm from below} at $\bar p\in P$ if $\bar p\in
\dom\phi$ and it holds
$$
    \liminf_{p\to\bar p}\frac{\phi(p)-\phi(\bar p)}{d(p,\bar p)}
    >-\infty.
$$
By replacing $\displaystyle\liminf_{p\to\bar p}$ and $-\infty$
with $\displaystyle \limsup_{p\to\bar p}$ and $+\infty$, respectively,
and by reversing the above inequality, one obtains the calmness from
above of $\phi$ at $\bar p$. Of course, $\phi$ is calm at $\bar p$ iff
it is both calm from above and from below at the same point.

So far the notion of calmness has been referred to mappings and
functionals. The next definition enlightens the variational nature
of calmness by referring it to perturbed optimization problems.

Let $\varphi:P\times X\longrightarrow\R$ and $h:P\times X
\longrightarrow Y$ be given functions and let $C$ be a nonempty
closed subset ot $Y$. The basic format of the parametric optimization
problems considered in what follows is defined as below:
$$
    \min_{x\in X}\varphi(p,x)  \ \hbox{ subject to } h(p,x)\in C  
    \leqno (\mathcal{P}_p)
$$
where $R(p)=\{x\in X:\ h(p,x)\in C\}$ is the feasible region and
$\varphi$ the objective functional.
The associated {\it optimal value} (or {\it marginal}) {\it function}
$\valf:P\longrightarrow\R\cup\{\pm\infty\}$ is
$$
    \valf(p)=\inf_{x\in  R(p)}\varphi(p,x).
$$
The associated global solution (minimizer) set, denoted by $\Argmin:P
\longrightarrow 2^X$, is the (generally) set-valued mapping
$$
    \Argmin(p)=\{x\in R(p):\ \varphi(p,x)=\valf(p)\}.
$$
Observe that in such format, the class of problems $(\mathcal{P}_p)$
needs no linear structure. However, for performing any quantitative
analysis of it, the parameter space $P$  must be endowed at least
with a metric structure.
Owing to its abstract formalism, $(\mathcal{P}_p)$ can cover large
classes of constrained (and unconstrained) optimization problems.
A particular case is the standard nonlinear programming problem:
in such event $X$ becomes a finite-dimensional Euclidean space
and the feasible region is given by the solution set of a system of
finitely many equalities/inequalities, namely
$R(p)=\{x\in X:\ h_i(p,x)=0,\ i=1,\dots,m_1; \ 
h_i(p,x)\le 0,\ i=m_1+1,\dots,m_1+m_2\}$.

\begin{definition}      \label{def:procalm}
With reference to a parametric family of problems $(\mathcal{P}_p)$,
let $\bar p\in P$ and let $\bar x\in R(\bar p)$ be a global
solution for $(\mathcal{P}_{\bar p})$. Problem $(\mathcal{P}_{\bar p})$
is said to be {\it calm} at $\bar x$ if there exists a positive
constant $r>0$ such that
$$
    \inf_{p\in \ball(\bar p,r)\backslash\{\bar p\}} \inf_{x\in R(p)
    \cap \ball(\bar x,r)}\ 
    \frac{\varphi(p,x)-\varphi(\bar x,\bar p)}{d(p,\bar p)}
    >-\infty.
$$
\end{definition}

Calmness for parametric constrained optimization problems and
calmness from below for functionals are readily linked via the value function,
as stated in the following proposition.

\begin{proposition}       \label{pro:pcalmfromcalm}
Let $\bar p\in P$ and let $\bar x\in X$ be a global solution for
$(\mathcal{P}_{\bar p})$. If the value function $\valf:P
\longrightarrow\R\cup\{\pm\infty\}$ associated to $(\mathcal{P}_p)$
is calm from below at $\bar p$ then problem $(\mathcal{P}_{\bar p})$
is calm at $\bar x$.
\end{proposition}

\begin{proof}
Notice that, being $\bar x\in\Argmin(\bar p)$, it is $\bar p
\in\dom\valf$. Then, it suffices to observe that for any
$p\in P$ and $r>0$, according to the definition of $\valf$,
one has
$$
   \inf_{x\in R(p)\cap \ball(\bar x,r)}\frac{\varphi(p,x)-
  \varphi(\bar x,\bar p)}{d(p,\bar p)}\ge
    \inf_{x\in R(p)}\frac{\varphi(p,x)-\varphi(\bar x,\bar p)}
   {d(p,\bar p)}=\frac{\valf(p)-\valf(\bar p)}{d(p,\bar p)}
$$
and to recall Definition \ref{def:procalm}, along with the definition
of calmness from below.
\end{proof}

It is worth mentioning that, in the particular case in which $X$
and $Y$ are normed vector spaces and the problem data take the
special form
$$
    \varphi(p,x)=\varphi(x)\qquad\hbox{and}\qquad h(p,x)=g(x)-p,
$$
the problem calmness of $(\mathcal{P}_{\bar p})$ at $\bar x\in
\Argmin(\bar p)$ is known to characterize the existence of penalty
functions (see \cite{Burk91}). The next proposition provides a
sufficient condition for $\valf$ to be calm from above.

\begin{proposition}
With reference to a parametric class of problem $(\mathcal{P}_p)$,
let $\bar p\in P$ and $\bar x\in\Argmin(\bar p)$. If the set-valued
mapping $R:P\longrightarrow 2^X$ is Lipschitz l.s.c. at $(\bar p,
\bar x)$ and $\varphi:P\times X\longrightarrow\R$ is calm from
above at $(\bar p,\bar x)$, then function $\valf:P\longrightarrow
\R\cup\{\pm\infty\}$ is calm from above at $\bar p$.
\end{proposition}

\begin{proof}
Since $\varphi$ is calm from above at $(\bar p,\bar x)$, there
exist positive $\delta$ and $\kappa$ such that, using the sum metric
on the space $P\times X$, it holds
\begin{eqnarray}    \label{in:facalobj}
   \varphi(p,x)-\varphi(\bar p,\bar x)\le\kappa[d(p,\bar p)+d(x,\bar x)],
  \quad\forall x\in\ball (\bar x,\delta),\ \forall p\in\ball (\bar p,\delta).
\end{eqnarray}
By the Lipschitz lower semicontinuity of $R$ at $(\bar p,\bar x)$ one gets the
existence of positive $\zeta$ and $l$ such that
$$
    R(p)\cap \ball(\bar x,ld(p,\bar p))\ne\varnothing,\quad
    \forall p\in\ball (\bar p,\zeta).
$$
This means that if $p\in\ball(\bar p,\zeta)$ there exists $x_p\in
R(p)$ such that
$$
    d(x_p,\bar x)\le ld(p,\bar p).
$$
Notice that without any loss of generality it is possible to assume
that $\zeta<\delta/l$. Such an assumption implies that  it is also $x_p\in
\ball(\bar x,\delta)$. From the above estimation of $d(x_p,\bar x)$,
by taking account of (\ref{in:facalobj}), one obtains
$$
    \frac{\valf(p)-\valf(\bar p)}{d(p,\bar p)}\le
    \frac{\varphi(p,x_p)-\varphi(\bar p,\bar x)}{d(p,\bar p)}\le
   \frac{\kappa[d(p,\bar p)+d(x_p,\bar x)]}{d(p,\bar p)}\le
   \kappa[1+l],\quad\forall p\in\ball (\bar p,\zeta)\backslash
   \{\bar p\}.
$$
Consequently, this allows to conclude that
$$
    \limsup_{p\to\bar p}\frac{\valf(p)-\valf(\bar p)}{d(p,\bar p)}
    \le\kappa[1+l]<+\infty.
$$
\end{proof}

A sufficient condition for $\valf$ to be calm from below is established
in the next proposition.

\begin{proposition}
With reference to a parametric class of problem $(\mathcal{P}_p)$,
let $\bar p\in P$ and $\bar x\in\Argmin(\bar p)$. If the set-valued
mapping $R:P\longrightarrow 2^X$ is upper Lipschitz at $\bar p$
and $\varphi$ is Lipschitz continuous on $P\times X$, then
$\valf:P\longrightarrow\R\cup\{\pm\infty\}$ is calm from below
at $\bar p$ and problem $(\mathcal{P}_{\bar p})$ is calm at
$\bar x$.
\end{proposition}

\begin{proof}
The Lipschitz continuity of $\varphi$ on $P\times X$, ensures
the existence of a positive $\kappa$ such that, by equipping
$P\times X$ with the metric sum, it holds
\begin{eqnarray}    \label{in:Lipobj}
    |\varphi(p,x_1)-\varphi(p,x_2)|\le\kappa [d(x_1,x_2)+d(p_1,p_2)],
    \quad\forall x_1,\, x_2\in X,\ \forall p_1,\, p_2\in P.
\end{eqnarray}
Since $R$ has been supposed to be upper Lipschitz continuous at
$\bar p$, there are $\ell\ge 0$ and $\tilde\zeta>0$
such that
$$
    \dist(x,R(\bar p))\le \ell d(p,\bar p),\quad\forall x\in R(p),\ 
    \forall p\in\ball (\bar p,\tilde\zeta).
$$
The above inequality means that  for every $x\in R(p)$ there is
$z_x\in R(\bar p)$ satisfying the property
$$
   d(z_x,x)\le (\ell+1)d(p,\bar p),
$$
provided that $p\in\ball (\bar p,\tilde\zeta)$. Therefore, by using
the metric sum in $P\times X$ and inequality (\ref{in:Lipobj}),
one obtains
$$
     \varphi(p,x)\ge\varphi(\bar p,z_x)-\kappa[d(p,\bar p)+d(x,z_x)]
     \ge \varphi(\bar p,\bar x)-\kappa[\ell+2]d(p,\bar p),\quad
     \forall x\in R(p),
$$
whence it follows
$$
   \frac{\valf(p)-\valf(\bar p)}{d(p,\bar p)}\ge
   \frac{\displaystyle\inf_{x\in R(p)}\varphi(p,x)-\varphi(\bar p,\bar x)}{d(p,\bar p)}
   \ge -\kappa[\ell+2],\quad\forall p\in\ball (\bar p,\zeta)\backslash
   \{\bar p\}.
$$
The second assertion of the thesis is a straightforward consequence
of the first one in the light of Proposition \ref{pro:pcalmfromcalm}.
\end{proof}

If relaxing the assumption of Lipschitz upper semicontinuity on $R$
using instead calmness, and the assumption of Lipschitz continuity on $\varphi$
using its local counterpart, one can still obtain a calmness behaviour
of the value function, yet in a weaker form. In order to introduce it,
given $\hat x\in X$, define function $\locvalf{\hat x}:P\times (0,+\infty)
\longrightarrow\R\cup\{\pm\infty\}$ as follows
$$
    \locvalf{\hat x}(p,r)=\inf_{x\in R(p)\cap \ball(\hat x,r)}\varphi(p,x).
$$

\begin{proposition}
With reference to a parametric class of problem $(\mathcal{P}_p)$,
let $\bar p\in P$ and $\bar x$ be a local minimizer for $(\mathcal{P}_{\bar p})$.
If the set-valued mapping $R:P\longrightarrow 2^X$ is calm at $(\bar p,\bar x)$
and $\varphi$ is locally Lipschitz near $(\bar p,\bar x)$, then there
exists $r_0>0$ such that $\locvalf{\bar x}(\cdot,r_0)$ is calm from below
at $\bar p$.
\end{proposition}

\begin{proof}
Since $\bar x$ is a local minimizer for $(\mathcal{P}_{\bar p})$,
for some $\hat\delta>0$ one has
$$
    \varphi(\bar p,x)\ge\varphi(\bar p,\bar x),\quad\forall x\in
    R(\bar p)\cap\ball (\bar x,\hat\delta).
$$
Being $\varphi$ locally Lipschitz around $(\bar p,\bar x)$, there
exist positive reals $\kappa$, $\tilde\delta$, and $\tilde\zeta$ such that
\begin{eqnarray}    \label{in:locLippx}
    |\varphi(p_1,x_1)-\varphi(p_2,x_2)|\le\kappa [d(p_1,p_2)+
       d(x_1,x_2)],\quad\forall p_1,\, p_2\in\ball (\bar p,\tilde\zeta),
      \ \forall x_1,\, x_2\in\ball (\bar x,\tilde\delta).
\end{eqnarray}
Since $R$ is calm at $(\bar p,\bar x)$, positive $\ell$, $\delta$
and $\zeta$ must exist such that
\begin{eqnarray}    \label{in:calmRlocvalf}
    \dist(x,R(\bar p))\le\ell d(p,\bar p),\quad\forall x\in R(p)
    \cap\ball(\bar x,\delta),\ \forall p\in\ball (\bar p,\zeta).
\end{eqnarray}
Without loss of generality, it is possible to assume
$$
   \delta<\min\left\{\frac{\tilde\delta}{3},\, \hat\delta\right\}
   \qquad \hbox{and}\qquad
   \zeta<\min\left\{\tilde\zeta,\, \frac{\tilde\delta}{3(\ell+1)}
   \right\}.
$$
By inequality (\ref{in:calmRlocvalf}) one has that for every $x\in
R(p)\cap\ball(\bar x,\delta)$ a $z_x\in R(\bar p)$ can be found such
that
$$
    d(z_x,x)\le(\ell+1)d(p,\bar p).
$$ 
Notice that, by virtue of the positions taken on $\delta$ and
$\zeta$, whenever $p\in\ball(\bar p,\zeta)$ and $x\in\ball(\bar x,
\delta)$ it results in
$$
    d(z_x,\bar x)\le d(z_x,x)+d(x,\bar x)\le (\ell+1)\zeta+\delta<
    \frac{2}{3}\tilde\delta<\tilde\delta.
$$
Thus, by exploiting inequality (\ref{in:locLippx}), one obtains
\begin{eqnarray*}
    \varphi(p,x) \ge \varphi(\bar p,z_x)-\kappa[d(p,\bar p)+d(x,z_x)]
   \ge \varphi(\bar p,\bar x)-\kappa[\ell+2]d(p,\bar p),\quad
   \forall x\in R(p)\cap\ball (\bar x,\delta),\ \forall p\in\ball (\bar p,\zeta),
\end{eqnarray*}
wherefrom it follows
$$
     \frac{\locvalf{\bar x}(p,\delta)-\locvalf{\bar x}(\bar p,\delta)}{d(p,\bar p)}
    \ge -\kappa[\ell+2]>-\infty,\quad\forall p\in\ball
   (\bar p,\zeta)\backslash\{\bar p\}.
$$
To complete the proof it suffices to set $r_0=\delta$.
\end{proof}


\subsection{Tools of generalized differentiation}

Let us start with introducing selected derivative-like tools, which enable
to carry out  a local variational analysis in metric spaces. The
basic one is the strong slope of a functional, originally proposed
in \cite{DeMaTo80} for quite different purposes. Given a function
$\varphi:X\longrightarrow\R\cup\{\pm\infty\}$, by {\it strong slope}
of $\varphi$ at $\bar x\in X$ the real-extended value
$$
   \stsl\varphi (\bar x)=\left\{ \begin{array}{ll}
                       0, & \qquad\text{if $\bar x$ is a local minimizer for $\varphi$},  \\
                       \displaystyle\limsup_{x\to\bar x} 
                       \frac{\varphi(\bar x)-\varphi(x)}{d(x,\bar x)}, & 
                       \qquad\text{otherwise.} 
                 \end{array}
   \right.
$$
is meant. The utility of its employment has been demonstrated
in several circumstances (see \cite{AzCoLu02,Ioff00}).
In the present paper the strong slope will be merely exploited as an element
to construct a more robust object, considering ``variational properties" of
$\varphi$ also at points of $\lev_{{}_{>\varphi(\bar x)}}\varphi$
near $\bar x$. This object is the {\it strict outer slope} of $\varphi$ at $\bar x$,
which is denoted by $\sostsl\varphi (\bar x)$ and defined by
$$
    \sostsl\varphi (\bar x)=\lim_{\epsilon\to 0^+}\inf
    \{\stsl\varphi (x): x\in \ball(\bar x,\epsilon),\ \varphi(\bar x)<
     \varphi(x)\le\varphi(\bar x)+\epsilon\}.
$$
The latter will be the key tool for formulating the main Lipschitz
semicontinuity criteria proposed in paper.
In order to estimate the strict outer slope in the more structured
setting of Banach spaces, one needs to deal with more refined
generalized differentiation constructions, namely subdifferentials
and coderivatives. Nonsmooth analysis abunds with notions of
this kind.  According to the features of the study here exposed,
the following dual object,
which is based on the Fr\'echet $\epsilon$-subdifferential and called
{\it strict outer $\epsilon$-subdifferential}, turns out to be
useful
$$
     \soesubdif\varphi(\bar x)=\bigcup\left\{\fesubdif\varphi(x):\ 
     x\in\ball(\bar x,\epsilon),\ \varphi(\bar x)<\varphi(x)\le
     \varphi(\bar x)+\epsilon\right\}
$$
where, given  $\epsilon\ge 0$, the set
$$
   \fesubdif\varphi(\bar x)=\left\{ x^*\in\X^*:\ \liminf_{x\to\bar x}
   \frac{\varphi(x)-\varphi(\bar x)-\langle x^*,x-\bar x\rangle}
   {\|x-\bar x\|}>-\epsilon\right\}
$$
denotes the Fr\'echet $\epsilon$-subdifferential (presubdifferential)
of $\varphi$ at $\bar x$. For $\epsilon=0$, one gets exactly the
canonical Fr\'echet (alias regular) subdifferential of $\varphi$ at $\bar x$ (in
\cite{BorZhu05,Mord06,Mord06b,RocWet98,Schi07} the related
theory, with examples and applications, is presented in full detail).

In a special class of Banach spaces the next subdifferential
construction, called {\it strict outer subdifferential slope} of
$\varphi$ at $\bar x$ and defined as follows
$$
    \soesubdifsl{\varphi}(\bar x)=\lim_{\epsilon\to 0^+}\inf
    \{\|x^*\|:\ x^*\in \soesubdif\varphi(\bar x)\},
$$
revealed to be a proper tool, in order to provide a characterization
of the stric outer slope.

Some useful facts concerning the strict outer slope of a functional and its dual
characterizations in terms of strict outer subdifferential slope
are collected in the following remark. Their proofs in some case are 
straightforward, in some other can be found in \cite{FaHeKrOu10,Ioff00,Schi07}.

\begin{remark}    \label{rem:stsllipder}

(i) Let $\varphi:X\longrightarrow\R\cup\{\pm\infty\}$ and
$\psi:X\longrightarrow\R$
given functions, with $\psi$ locally Lipschitz near $x\in\dom\varphi$, having
Lipschitz constant $k$. Then one has
$$
    \stsl{(\varphi+\psi)} (x)\ge \stsl \varphi (x)-k.
$$
Consequently, it follows
$$
    \sostsl{(\varphi+\psi)} (x)\ge \sostsl \varphi (x)-k.
$$

(ii) If $(\X,\|\cdot\|)$ is an Asplund space and $\varphi:\X
\longrightarrow\R\cup\{\pm\infty\}$ is l.s.c. near $\bar x\in\dom
\varphi$, the following exact estimation holds true
$$
    \sostsl\varphi (\bar x)=\soesubdifsl{\varphi}(\bar x).
$$
Let us recall that a Banach space is said to be {\it Asplund} provided
that each of its separable subspaces admits separable dual.
It is to be noted that, with the development of variational and
extremal principles, alternative characterizations of the notion of
Asplundity have been discovered, which are expressed in terms of
dense Fr\'echet differentiability or dense Fr\'echet subdifferentiability
properties of continuous convex or merely l.s.c. functions,
respectively (see \cite{Mord06,Schi07}). Even though it excludes
certain Banach spaces (for example, the space $\ell_1(\N)$), the
class of Asplund spaces is sufficiently rich for major applications.
For instance, it includes all weakly compactly generated spaces,
and hence all reflexive spaces.

Again, the above assumptions on $\X$ and $\varphi$ enable one to exploit a simpler
representation for the strict outer $\epsilon$-subdifferential
as follows
$$
     \soesubdif\varphi(\bar x)=\bigcup\left\{\widehat\partial
     \varphi(x):\ x\in\ball(\bar x,\epsilon),\ \varphi(\bar x)
    <\varphi(x)\le\varphi(\bar x)+\epsilon\right\}.
$$

(iii) It is well known that, whenever $\varphi$ is Fr\'echet
differentiable at $\bar x$, with derivarite $\FrDer\varphi
(\bar x)\in\X^*$, it results in
$$
   \stsl{\varphi}(\bar x)=\|\FrDer\varphi(\bar x)\|.
$$
Analogously, whenever  $\varphi$ is strictly differentiable at
$\bar x$, with strict derivarite $\StDer\varphi(\bar x)\in\X^*$,
it holds
$$
    \sostsl{\varphi}(\bar x)=\|\StDer\varphi(\bar x)\|.
$$

\end{remark}


It is well known that for the canonical Fr\'echet subdifferential
a nonconvex counterpart of the Moreau-Rockafellar rule generally
fails to hold. Nonetheless, in Asplund spaces such a rule can be
restored in an approximate form, called fuzzy sum rule. In view
of its employment in a subsequent section, it is recalled in the
next lemma (see \cite{BorZhu05,Ioff00,Mord06,Schi07}).

\begin{lemma}     \label{lem:fuzsumrule}
Let $(\X,\|\cdot\|)$ be an Asplund space, let $\psi_1:\X\longrightarrow
\R$ and $\psi_2:\X\longrightarrow\R\cup\{\pm\infty\}$ given functions,
and let $\bar x\in\X$. Suppose $\psi_1$ to be Lipschitz continuous
around $\bar x$ and $\psi_2$ to be l.s.c. in a neighbourhood of
$\bar x\in\dom\psi_2$. Then, for any $\eta>0$ there
exist $x_1,\, x_2\in \ball(\bar x,\eta)$ such that
$$
     |\psi_i(x_i)-\psi_i(\bar x)|\le\eta,\quad i=1,\, 2,
$$
and
$$
    \fsubdif(\psi_1+\psi_2)(\bar x)\subseteq \fsubdif\psi_1(x_1)
   +\fsubdif\psi_2(x_2)+\eta\uball^*.
$$
\end{lemma}

When dealing with sets and set-valued mappings, one needs further
elements of Fr\'echet nonsmooth calculus. Given a subset $\Omega
\subseteq\X$ and $\bar x\in\Omega$, the set of all the Fr\'echet normals to
$\Omega$ at $\bar x$, denoted by
$$
   \Normal(\bar x,\Omega)=\left\{x^*\in\X^*:\ \limsup_{\Omega\atop
    \displaystyle x\to\bar x}\frac{\langle x^*,x-\bar x\rangle}{\|x-\bar x\|}
   \le 0\right\},
$$
is called {\it Fr\'echet normal cone} to $\Omega$ at $\bar x$.
The Fr\'echet normal cone to a set is linked with the Fr\'echet
subdifferential via the set indicator function, as clarified by the
equality
$$
   \Normal(\bar x,\Omega)=\fsubdif\iota(\cdot,\Omega)(\bar x),
    \quad \bar x\in\X.
$$
In order to define derivatives of set-valued mappings, the graphical
approach relying on the notion of coderivative, which is due to
B.S. Mordukhovich (see \cite{Mord80}), is here adopted.
Given a multifunction $\Phi:\X\longrightarrow 2^\Y$ between
Banach spaces and $(\bar x,\bar y)\in\grph\Phi$, by {\it Fr\'echet
coderivative} of $\Phi$ at $(\bar x,\bar y)$ the set-valued mapping
$\Coder\Phi:\Y^*\longrightarrow 2^{\X^*}$, defined through the
Fr\'echet normal cone as follows
$$
     \Coder\Phi(\bar x,\bar y)(y^*)=\{x^*\in\X^*:\ (x^*,-y^*)\in\Normal
     ((\bar x,\bar y),\grph\Phi)\},\quad y^*\in\Y^*,
$$
is meant.
Being a positively homogeneous mapping, one can consider its
{\it outer norm}, i.e.
$$
    \|\Coder\Phi(\bar x,\bar y)\|_+=\sup_{\|y^*\|=1}
    \sup\{\|x^*\|:\ x^*\in\Coder\Phi(\bar x,\bar y)(y^*)\},
$$
which will be useful in a subsequent section.


The next lemma, whose proof can be found for instance in \cite{Uder09},
will be employed as well in the sequel.

\begin{lemma}       \label{lem:displsc}
Let $f:X\longrightarrow Y$ and $F:X\longrightarrow 2^Y$ be mappings
between metric spaces and let $\bar x\in X$. If $f$ is continuous
at $\bar x$ and $F$ is u.s.c. at the same point, then function
$d[f,F]:X\longrightarrow [0,+\infty]$ defined by
$$
   d[f,F](x)=\dist(f(x),F(x)),
$$
is l.s.c. at $\bar x$.
\end{lemma}


\section{Lipschitz lower semicontinuity of variational systems}    \label{sec:Liplsccriter}

In order to denote a function measuring displacements of the values
of $f$ from those of $F$, the following shortened notation is introduced
$$
   \pdisp(p,x)=\dist(f(p,x),F(p,x)).
$$
For function $\pdisp:P\times X\longrightarrow [0,+\infty]$,
a partial version of the strict outer slope, which will be employed
in the statement of the next result, is defined as follows
\begin{eqnarray*}
   \sostslx{\pdisp}(\bar p,\bar x) = \lim_{\epsilon\to 0^+}\inf
    \{\stslx{\pdisp}(p,x):\  & & x\in \ball(\bar x,\epsilon),\ 
   p\in \ball(\bar p,\epsilon),\  \\
   & &\pdisp(\bar p,\bar x)<\pdisp
  (p,x)\le\pdisp(\bar p,\bar x)+\epsilon\},
\end{eqnarray*}
where
$$
    \stslx{\pdisp}(p,x)=\left\{ \begin{array}{ll}
                       0, & \qquad\text{if $x$ is a local minimizer for $\pdisp(p,\cdot)$},  \\
                       \displaystyle\limsup_{z\to x} 
                       \frac{\pdisp(p,x)-\pdisp(p,z)}{d(z,x)}, & 
                       \qquad\text{otherwise.} 
                 \end{array}
   \right.
$$

One is now in a position to formulate the following sufficient
condition for Lischitz lower semicontinuity of variational systems,
which is valid in metric spaces.

\begin{theorem}      \label{thm:impLiplsc}
Let $f:P\times X\longrightarrow Y$ and $F:P\times X\longrightarrow
2^Y$ be mappings defining a problem $(\mathcal{GE}_p)$, with
solution mapping $G:P\longrightarrow 2^X$. Let $\bar p\in P$,
$\bar x\in G(\bar p)$ and $\bar y=f(\bar p,\bar x)$. Suppose that

(i) $(X,d)$ is metrically complete;

(ii) there exists $\delta_1>0$ such that, for every $p\in \ball(\bar p,
\delta_1)$, mapping $F(p,\cdot):X\longrightarrow 2^Y$ is u.s.c. at
each point of $\ball(\bar x,\delta_1)$;

(iii) there exist $\delta_2>0$ and $\l_F>0$ such that
$$
    F(p,\bar x)\cap \ball(\bar y,l_F d(p,\bar p))\ne\varnothing,\quad
    \forall p\in \ball(\bar p,\delta_2);
$$

(iv) there exists $\delta_3>0$ such that, for every $p\in \ball(\bar p,
\delta_3)$, $f(p,\cdot)$ is continuous at each point of $\ball(\bar x,
\delta_3)$;

(v) there exist $\delta_4>0$ and $\l_f>0$ such that
$$
     d(f(p,\bar x),f(\bar p,\bar x))\le l_f d(p,\bar p),\quad
    \forall p\in \ball(\bar p,\delta_4);
$$

(vi) it holds
\begin {eqnarray*}
    \sostslx\pdisp (\bar p,\bar x)>0.
\end{eqnarray*}

\noindent Then $G$ is Lipschitz l.s.c. at $(\bar p,\bar x)$.
\end{theorem}

\begin{proof}
Let $\delta=\min\{\delta_1,\, \delta_2,\,\delta_3,\delta_4\}$, where
each $\delta_i$, with $i=1,2,3,4$,
is as stated in hypotheses (ii), (iii), (iv), and (v), respectively.
According to hypothesis (vi), it is possible to fix $c$ and $c'$ in
such a way that
$$
   0<c<c'<\sostslx\pdisp (\bar p,\bar x).
$$
Corresponding to $c'$, there exists $\delta_*\in (0,\delta)$ such
that
\begin{eqnarray}    \label{in:stslpdisp}
    \stslx\pdisp (p,x)>c',\quad\forall p\in\ball(\bar p,\delta_*),\
    \forall x\in\ball(\bar x,\delta_*) \hbox{ with } 0<\pdisp(p,x)
    \le\delta_*.
\end{eqnarray} 
This amounts to say that for every $(p,x)\in\ball(\bar p,\delta_*)
\times\ball(\bar x,\delta_*)$, with $f(p,x)\not\in F(p,x)$ and $\pdisp
(p,x)\le\delta_*$, and for every $\eta>0$ an element $x_\eta\in
\ball(x,\eta)$ must exist such that
$$
   \dist(f(p,x),F(p,x))>\dist(f(p,x_\eta),F(p,x_\eta))+c'd(x_\eta,x).
$$
Now, choose a positive $\zeta_c$ in such a way that
$$
    \zeta_c<\min\left\{\frac{1}{2},\frac{c}{2 (l_F+l_f+1)},
    \frac{1}{l_F+l_f+1}\right\}\delta_*,
$$
and fix $p\in \ball(\bar p,\zeta_c)\backslash\{\bar p\}$. By virtue
of hypothesis (iii), being $\zeta_c<\delta_2$, it holds
$$
    F(p,\bar x)\cap \ball(\bar y,l_F\, d(p,\bar p))\ne\varnothing.
$$
This means that there is $v\in F(p,\bar x)$ such that
\begin{eqnarray}    \label{in:vbary}
     d(v,\bar y)\le l_F\, d(p,\bar p).
\end{eqnarray}
Notice that, being $\zeta_c<\delta_1$ and $\zeta_c<\delta_3$,
$f(p,\cdot)$ is continuous on $\ball(\bar x,\delta_*)$ as well as
$F(p,\cdot)$ is u.s.c. on the same subset.
Thus, according to Lemma \ref{lem:displsc}, if restricted to the
complete metric space  $\ball(\bar x,\delta_*)$, the displacement
function $\pdisp(p,\cdot)$ turns out to be l.s.c.. Moreover, by
virtue of hypotheis (v) and inequality (\ref{in:vbary}), being
$\zeta_c<\delta_4$, one has
\begin{eqnarray*}
    \pdisp(p,\bar x) &=& \inf_{w\in F(p,\bar x)}d(f(p,\bar x),w)\le
    d(f(p,\bar x),\bar y)+d(\bar y,v)\le (l_f+l_F)d(p,\bar p) \\
    &\le &\inf_{x\in\ball(\bar x,\delta_*)}\pdisp(p,x)+(l_f+l_F)d(p,\bar p).
\end{eqnarray*}
By invoking the Ekeland variational principle, with  $\lambda=
\frac{l_f+l_F}{c'}d(p,\bar p)$, one gets the existence of $\hat x\in
\ball(\bar x,\delta_*)$, such that
\begin{eqnarray}     \label{in:evpbis1}
    \pdisp(p,\hat x)=\dist(f(p,\hat x),F(p,\hat x))\le\pdisp(p,\bar x)\le
    (l_f+l_F)\, d(p,\bar p),
\end{eqnarray}
\begin{eqnarray}     \label{in:evpbis2}
     d(\hat x,\bar x)\le\frac{l_f+l_F}{c'}d(p,\bar p)
    \le\frac{l_f+l_F}{c'}\zeta_c<
     \frac{c}{2c'}\delta_*<\frac{\delta_*}{2},
\end{eqnarray}
and
\begin{eqnarray}     \label{in:evpbis3}
     \pdisp(p,\hat x)<\pdisp(p,x)+c'd(x,\hat x),\quad\forall x\in\ball
    (\bar x,\delta_*)\backslash\{\hat x\}.
\end{eqnarray}
Inequality (\ref{in:evpbis3}) allows to say that
\begin{eqnarray}      \label{inc:hatxFphatx}
      f(p,\hat x)\in F(p,\hat x).
\end{eqnarray}
Indeed, assume the contrary. Since in particular $\hat x\in
\ball(\bar x,\delta_*)$ and $p\in \ball(\bar p,\delta_*)$ and, by the last
inequality in (\ref{in:evpbis1}), it is $0<\pdisp(p,\hat x)<\delta_*$,
then one is enabled to apply what has been obtained as a consequence
of (\ref{in:stslpdisp}).
Accordingly, corresponding to $\eta=\delta_*/2$, an element $x_\eta\in
\ball(\hat x,\delta_*/2)$ must exist such that
\begin{eqnarray}     \label{in:contrevpbis3}
   \dist(f(p,\hat x),F(p,\hat x))>d(f(p,x_\eta),F(p,x_\eta))
   +c'd(x_\eta,\hat x).
\end{eqnarray}
Notice that, since
$$
    d(x_\eta,\bar x)\le d(x_\eta,\hat x)+d(\hat x,\bar x)<\frac{\delta_*}{2}+
    \frac{\delta_*}{2}=\delta_*,
$$
then $x_\eta\in\ball(\bar x,\delta_*)$. In the light of this, inequality (\ref{in:contrevpbis3})
becomes inconsistent with (\ref{in:evpbis3}), thereby validating
inclusion (\ref{inc:hatxFphatx}). It follows that $\hat x\in G(p)$. 
Thus, being $\hat x\in \ball(\bar x,\frac{l_f+l_F}{c}\, d(p,\bar p))$ by
virtue of (\ref{in:evpbis2}), it has been shown that
$$
    \hat x\in G(p)\cap\ball\left(\bar x,\frac{l_f+l_F}{c} d(p,\bar p)\right)
    \ne\varnothing.
$$
The proof is now complete.
\end{proof}

On the base of Remark \ref{rem:Liplsccalmgen}, the reader should
notice that Theorem \ref{thm:impLiplsc} guarantees, in particular,
nonempty-valuedness of variational systems near a point of interest.
This fact seems to be notable in consideration of the bizarre geometry
of solution mappings to parameterized generalized equations,
whose graphs often exhibit ``corner configurations": near them
$G$ passes from multivaluedness for some values of $p$ to
emptiness, for some others.


\section{Calmness of variational systems}     \label{sec:calmcriter}

In order to study the calmness property of variational systems,
a further displacement function $\disp:X\times Y\longrightarrow
\R\cup\{\pm\infty\}$, defined as
$$
    \disp(x,y)=d(f(\bar p,x),y)+\iota((x,y),\grph F(\bar p,\cdot)),
$$
is exploited. It enters the statement of the next result.

\begin{theorem}       \label{thm:impcalm}
Let $f:P\times X\longrightarrow Y$ and $F:P\times X\longrightarrow 2^Y$ be
mappings defining a problem $(\mathcal{GE}_p)$, with solution
mapping $G:P\longrightarrow 2^X$. Let $\bar p\in P$, $\bar x\in
G(\bar p)$ and $\bar y=f(\bar p,\bar x)$. Suppose that:

(i) $(X,d)$ and $(Y,d)$ are metrically complete;

(ii) there exists $\delta_1>0$ such that
$[\ball(\bar x,\delta_1)\times\ball(\bar y,\delta_1)]\cap
    \grph F(\bar p,\cdot)$
is a closed subset of $X\times Y$;

(iii) there exist $\delta_2>0$ and $\ell_F>0$ such that
$$
    F(p,x)\subseteq \ball\left(F(\bar p,x),\ell_F d(p,\bar p)\right),
    \quad\forall x\in \ball(\bar x,\delta_2),\ 
    \forall p\in \ball(\bar p,\delta_2);
$$

(iv) there exist $\delta_3>0$ such that $f(\bar p,\cdot)$ is
continuous at each point of $\ball(\bar x,\delta_3)$;

(v) there exist $\delta_4>0$ and $\ell_f>0$ such that
$$
    d(f(p,x),f(\bar p,x))\le \ell_f d(p,\bar p),\quad\forall
   x\in \ball(\bar x,\delta_4),\ 
    \forall p\in \ball(\bar p,\delta_4);
$$

(vi) it holds
\begin {eqnarray*}
    \sostsl\disp (\bar x,\bar y)>0.
\end{eqnarray*}

\noindent Then $G$ is calm at $(\bar p,\bar x)$ and the following
estimation holds
$$
    \calm G(\bar p,\bar x)\le\frac{\ell_f+\ell_F}{\sostsl\disp
   (\bar x,\bar y)}.
$$
\end{theorem}

\begin{proof}
To prove both the assertions in the thesis it suffices to show
that for every positive $c$, with $c<\sostsl\disp (\bar x,\bar y)$
there exists $\zeta_c>0$ with the property
\begin{eqnarray}       \label{inc:Gcalm}
   G(p)\cap \ball(\bar x,\zeta_c)\subseteq \ball\left(G(\bar p),
    \frac{\ell_f+\ell_F}{c}\, d(p,\bar p)\right),\quad
   \forall p\in \ball(\bar p,\zeta_c).
\end{eqnarray}
So fix $c$ and $c'$ in such a way that
$$
   0<c<c'<\sostsl\disp (\bar x,\bar y).
$$
Set $\delta=\min\{\delta_1,\,\delta_2,\,\delta_3,\,\delta_4\}$. Then, by virtue
of hypothesis (vi) there exists $\delta_*\in (0,\delta)$ such that
$$
    \stsl\disp(x,y)>c',\quad\forall (x,y)\in [\ball(\bar x,\delta_*)\times 
     \ball(\bar y,\delta_*)]\cap\grph F(\bar p,\cdot)\hbox{ with }
        y\ne f(\bar p,x).
$$
This means, in particular, that for every $(x,y)\in [\ball(\bar x,
\delta_*)\times \ball(\bar y,\delta_*)]\cap\grph F(\bar p,\cdot)$
with $y\ne f(\bar p,x)$ and $\eta>0$ there is $(x_\eta,y_\eta)\in
[\ball(x,\eta)\times \ball(y,\eta)]
\cap\grph F(\bar p,\cdot)$ such that
\begin{eqnarray*}
     d(f(\bar p,x),y)>d(f(\bar p,x_\eta),y_\eta)+c'd((x_\eta,y_\eta),(x,y)).
\end{eqnarray*}
According to hypothesis (iv), corresponding to $\delta_*/16$ there
exists $\tilde\delta>0$ such that
\begin{eqnarray*}
   d(f(\bar p,z),\bar y)<\frac{\delta_*}{16},\quad\forall z\in
   \ball(\bar x,\tilde\delta).
\end{eqnarray*}
Now, take a positive $\zeta_c$ as follows
\begin{eqnarray}     \label{in:defzetac}
     \zeta_c<\min\left\{\frac{\delta_*}{16},\, \frac{c\delta_*}{2(\ell_f+2\ell_F+1)},\,
     \frac{\delta_*}{16(\ell_F+1)},\, \frac{\delta_*}{16(\ell_f+1)},
     \, \frac{\delta_*}{\ell_f+2\ell_F+1},\,\tilde\delta\right\}
\end{eqnarray}
and fix $p\in\ball(\bar p,\zeta_c)\backslash\{\bar p\}$. If $G(p)
\cap \ball(\bar x,\zeta_c)=\varnothing$, then inclusion (\ref{inc:Gcalm})
trivially holds true. Otherwise, take an arbitrary $z\in G(p)
\cap\ball(\bar x,\zeta_c)$. Then
$$
    d(z,\bar x)\le\zeta_c
$$
and
\begin{eqnarray}     \label{inc:zsolpGE}
   f(p,z)\in F(p,z).
\end{eqnarray}
Notice that, due to position (\ref{in:defzetac}), it is $z\in\ball(\bar x,\delta_2)$
and $p\in\ball(\bar p,\delta_2)$ so, on the account of hypothesis (iii),
one obtains from inclusion (\ref{inc:zsolpGE})
$$
   f(p,z)\in\ball(F(\bar p,z),\ell_F d(p,\bar p)).
$$
The above inclusion entails the existence of $w\in F(\bar p,z)$
such that
$$
   d(f(p,z),w)\le (1+\epsilon)\ell_F d(p,\bar p),
$$
where $\epsilon$ can be assumed to fulfil the inequalities
\begin{eqnarray}     \label{in:epsmin}
   0<\epsilon<\min\left\{1,\frac{(c'-c)(\ell_f+\ell_F)}{c\ell_F}
    \right\}.
\end{eqnarray}
Thus, in the light of position (\ref{in:defzetac}), it results in
$$
   d(f(p,z),w)<2\ell_F\zeta_c<\frac{\delta_*}{8}.
$$
Since in particular $z\in \ball(\bar x,\delta_4)$ and $z\in
\ball(\bar x,\tilde\delta)$, it follows
\begin{eqnarray}       \label{in:wbary}
   d(w,\bar y)\le d(w,f(p,z))+d(f(p,z),f(\bar p,z))+d(f(\bar p,z),\bar y)
  <\frac{\delta_*}{8}+ \frac{\delta_*}{16}+\frac{\delta_*}{16}
  =\frac{\delta_*}{4}
\end{eqnarray}
and hence, a fortiori, it turns out that $(z,w)\in\ball(\bar x,\delta_*)
\times\ball(\bar y,\delta_*)$. Therefore, if restricting function $\disp$ to the
complete metric space $\ball(\bar x,\delta_*)\times\ball
(\bar y,\delta_*)$, one obtains 
\begin{eqnarray*}
    \disp(z,w) &\le & d(f(\bar p,z),f(p,z))+d(f(p,z),w)\le(\ell_f+(1+\epsilon)\ell_F)
     d(p,\bar p)    \\ 
    & \le &  \inf_{(x,y)\in \ball(\bar x,\delta_*)\times \ball(\bar y,\delta_*)}
    \disp(x,y)+(\ell_f+(1+\epsilon)\ell_F) d(p,\bar p)
\end{eqnarray*}
(recall that $(z,w)\in\grph F (\bar p,\cdot)$). Since $f(\bar p,
\cdot)$ is continuous
on $\ball(\bar x,\delta_*)$ and $[\ball(\bar x,\delta_*)\times 
\ball(\bar y,\delta_*)]\cap\grph F(\bar p,\cdot)$ is closed by
virtue of hypothesis (ii), function
$\disp$ turns out to be l.s.c. on $\ball(\bar x,\delta_*)\times 
\ball(\bar y,\delta_*)$. It is then possible to apply the Ekeland
variational principle, according to which, corresponding
to 
$$
    \lambda=\frac{\ell_f+(1+\epsilon)\ell_F}{c'}d(p,\bar p),
$$
there exists $(\hat x,\hat y)\in \ball(\bar x,\delta_*)\times
\ball(\bar y,\delta_*)$ such that
\begin{eqnarray}    \label{in:evp1}
   \disp(\hat x,\hat y)\le\disp(z,w)\le [\ell_f+(1+\epsilon)\ell_F]d(p,\bar p),
\end{eqnarray}
\begin{eqnarray}    \label{in:evp2}
    d((\hat x,\hat y),(z,w))\le\frac{\ell_f+(1+\epsilon)\ell_F}{c'}d(p,\bar p)<
    \frac{\ell_f+2\ell_F}{c'}\zeta_c<\frac{c\delta_*}{2c'}<\frac{\delta_*}{2},
\end{eqnarray}
and
\begin{eqnarray}    \label{in:evp3}
   \disp(\hat x,\hat y) &<& \disp(x,y)+c'd((x,y),(\hat x,\hat y)), \\
   & &\forall (x,y)\in
   [\ball(\bar x,\delta_*)\times \ball (\bar y,\delta_*)]
   \backslash\{(\hat x,\hat y)\}.   \nonumber
\end{eqnarray}
Notice that, since it is $\disp(\hat x,\hat y)<+\infty$ as a consequence
of (\ref{in:evp1}), it must be $\hat y\in F(\bar p,\hat x)$, and
hence it results in
$$
    \disp(\hat x,\hat y)=d(f(\bar p,\hat x),\hat y).
$$
Consequently, inequality (\ref{in:evp3}) takes the form
\begin{eqnarray}    \label{in:evp3bis}
   d(f(\bar p,\hat x),\hat y) < d(f(\bar p,x),y)+c'd((x,y),(\hat x,\hat y)),
    \quad\forall (x,y)\in
   [(\ball (\bar x,\delta_*)\times \ball (\bar y,\delta_*))
   \backslash\{(\hat x,\hat y)\}]\cap\grph F(\bar p,\cdot).
\end{eqnarray}
Moreover, being $\zeta_c<\frac{\delta_*}{\ell_f+2\ell_F+1}$,
one finds
$$
     \disp(\hat x,\hat y)=d(f(\bar p,\hat x),\hat y)\le
      [\ell_f+(1+\epsilon)\ell_F]d(p,\bar p)<\delta_*.
$$
The last inequalities lead to conclude that
\begin{eqnarray}       \label{inc:hatxsolpGE}
       f(\bar p,\hat x)\in F(\bar p,\hat x).
\end{eqnarray}
Indeed, assume ab absurdo that $f(\bar p,\hat x)\not\in F(\bar p,\hat x)$
and hence $\hat y\ne f(\bar p,\hat x)$. Since, by inequalities (\ref{in:evp2})
and (\ref{in:wbary}) it holds
$$
    d(\hat x,\bar x)\le d(\hat x,z)+d(z,\bar x)<\frac{\delta_*}{2}+
    \frac{\delta_*}{16}<\frac{3}{4}\delta_*
$$
and
$$
    d(\hat y,\bar y)\le d(\hat y,w)+d(w,\bar y)<\frac{\delta_*}{2}+
    \frac{\delta_*}{4}=\frac{3}{4}\delta_*,
$$
actually it is $(\hat x,\hat y)\in \ball(\bar x,\delta_*)\times
\ball(\bar y,\delta_*)$. Moreover $(\hat x,\hat y)\in\grph
F(\bar p,\cdot)$, so the consequence of hypothesis (vi) applies:
corresponding to $\eta=\delta_*/4$ an element $(x_\eta,y_\eta)
\in [\ball(\hat x,\delta_*/4)\times \ball(\hat y,\delta_*/4)]\cap\grph
 F(\bar p,\cdot)$ must exist such that
\begin{eqnarray}    \label{in:contrevp3bis}
    d(f(\bar p,\hat x),\hat y)>d(f(\bar p,x_\eta),y_\eta)+c'd((x_\eta,y_\eta),
    (\hat x,\hat y)).
\end{eqnarray}
Since
$$
     d(x_\eta,\bar x)\le d(x_\eta,\hat x)+d(\hat x,\bar x)<
     \frac{\delta_*}{4}+\frac{3}{4}\delta_*=\delta_*
$$
and
$$
     d(y_\eta,\bar y)\le d(y_\eta,\hat y)+d(\hat y,\bar y)<
     \frac{\delta_*}{4}+\frac{3}{4}\delta_*=\delta_*,
$$
then the existence of $(x_\eta,y_\eta)\in [(\ball(\bar x,\delta_*)
\times \ball(\bar y,\delta_*))\backslash\{(\hat x,\hat y)\}]\cap
\grph F(\bar p,\cdot)$
fulfilling inequality (\ref{in:contrevp3bis}) clearly contradicts
(\ref{in:evp3bis}).

Thus, inclusion (\ref{inc:hatxsolpGE})  means that
$\hat x\in G(\bar p)$. The fact that, according to the first
inequality in (\ref{in:evp2}) and (\ref{in:epsmin}), it holds
$$
   d(z,\hat x)\le\frac{(\ell_f+(1+\epsilon)\ell_F)}{c'}\, d(p,\bar p)
   <\frac{(\ell_f+\ell_F)}{c}d(p,\bar p),
$$
shows that $z\in\ball\left(G(\bar p),\frac{(\ell_f+\ell_F)}{c}
d(p,\bar p)\right)$, so that inclusion (\ref{inc:Gcalm}) appears
to be satisfied. This completes the proof.
\end{proof}

It has been remarked by several authors (see, for instance,
\cite{HenOut05}) that existent conditions for calmness actually
imply Aubin continuity. The following example shows that this
is not true for the condition proposed in Theorem \ref{thm:impcalm},
which thereby turns out to be a specific tool for detecting
calmness in circumstances where its stronger variant fails.

\begin{example}
Let $P=X=Y=\R$ be endowed with its usual metric structure. Consider
the variational system associated with the parameterized generalized
equation having base $f\equiv 0$ and field $F(p,x)=\{y\in\R:\ |y|\ge
|px|\}$. By explicitly resolving the inclusion $0\in F(p,x)$ one readily
finds
$$
    G(p)=\left\{\begin{array}{ll}
             \R, & \hbox{if } p=0, \\
             \{0\}, & \hbox{otherwise.}
         \end{array}\right.  
$$
So, choosing $\bar p=\bar x=0$ as reference points, one has
$\bar x\in G(\bar p)$.
One gets $\grph F(0,\cdot)=\R\times\R$, which is a closed set.
Besides, being polyhedral, mapping $F$ is upper Lipschitz
(uniformly in $x$) at $0$, with any $\ell_F>0$. Indeed, trivially
it holds
$$
    F(p,x)\subseteq\R=\ball(\R,0\cdot |p|)=\ball(F(0,x),\ell_Fd(p,0)),
    \quad\forall p\in\R,\ \forall x\in\R.
$$
Since, in the case under examination, it is
$$
    \disp(x,y)=d(0,y)+\iota((x,y),\grph F(0,\cdot))=|y|,
$$
one obtains that $0<\disp(x,y)\le\delta$ iff
$(x,y)\in\ball((0,0),\delta)\backslash\{(x,y)\in\R\times\R:\ y=0\}$.
As one checks at once, whenever $(x,y)\in \ball((0,0),\delta)
\backslash\{(x,y)\in\R\times\R:\ y=0\}$ it results in
$$
     \stsl{d(0,F(0,\cdot)}(x,y)=1,
$$
and therefore $\sostsl{d(0,F(0,\cdot)}(0,0)=1>0$. According to
Theorem \ref{thm:impcalm} the variational system $G$ is calm at
$(0,0)$ with $\calm G(0,0)=0$, what one can verify by using the
definition of $G$ explicitly derived. By direct inspection, it is
not difficult to see that $G$ fails to be Aubin continuous at
$(0,0)$.
\end{example}

Whenever $X=\X$ and $Y=\Y$ are Asplund spaces, from 
Theorem \ref{thm:impcalm} it is possible to derive an implicit
multifunction theorem establishing the calmness of variational
systems associated to parameterized generalized equations
with constant (null) base, i.e. having the special form
$$
    \nullv\in F(p,x),    \leqno(\mathcal{GE}^0_p),
$$
under a non degeneration condition on the Fr\'echet coderivative
of $F$. Such result seems to be rather close to the one achieved in \cite{ChKrYa11}
(see Theorem 3.1 therein). To see this fact in detail, given a solution
$\bar x\in G(\bar p)$, with $\bar p\in P$, let us set
\begin{eqnarray*}
   c[F(\bar p,\cdot)](\bar x,\nullv)=\lim_{\epsilon\to 0^+}\inf
   \{  \|x^*\|:\ x^*\in\Coder F(\bar p,\cdot)(x,y)(y^*),\ 
       x\in\ball(\bar x,\epsilon),\ y\in\ball(\nullv,\epsilon)
               \backslash\{\nullv\}, \\
        (x,y)\in\grph F(\bar p,\cdot),\  \|y^*\|=1\}.
\end{eqnarray*}

\begin{proposition}     \label{pro:calmAspcod}
Let $F:P\times\X\longrightarrow 2^\Y$ be a set-valued mapping defining
a problem $(\mathcal{GE}^0_p)$, with solution mapping $G:P
\longrightarrow 2^\X$, and let $\bar x\in G(\bar p)$. Suppose that:

(i) $(\X,\|\cdot\|)$ and $(\Y,\|\cdot\|)$ are Asplund;

(ii) the graph of $F(\bar p,\cdot):\X\longrightarrow 2^\Y$ is closed
in a neighbourhood of $(\bar x,\nullv)$;

(iii) there exist $\delta>0$ and $\ell>0$ such that
$$
    F(p,x)\subseteq\ball(F(\bar p,x),\ell d(p,\bar p)),\quad
    \forall x\in\ball(\bar x,\delta),\quad\forall p\in\ball
    (\bar p,\delta);
$$

(iv) it holds
\begin{eqnarray}     \label{in:condcoder}
    c[F(\bar p,\cdot)](\bar x,\nullv)>0.
\end{eqnarray}
Then, $G$ is calm at $(\bar p,\bar x)$.
\end{proposition}

\begin{proof}
The thesis follows as a consequence of Theorem \ref{thm:impcalm},
upon having proved that condition (\ref{in:condcoder}) implies
$$
   \sostsl\dispnull (\bar x,\nullv)>0.
$$
According to the definition of strict outer slope and to its subdifferential
representation (see Remark \ref{rem:stsllipder}(ii)), one has
\begin{eqnarray*}
    \sostsl\dispnull (\bar x,\nullv)=\lim_{\epsilon\to 0^+}\inf\{
    \|(x^*,y^*)\|:\ (x^*,y^*)\in\fsubdif\dispnull (x,y),
    \ (x,y)\in \ball (\bar x,\epsilon)\times\ball (\nullv,\epsilon), \\
   0<\|y\|+\iota((x,y),\grph F(\bar p,\cdot))\le\epsilon\}.
\end{eqnarray*}
By virtue of the hypothesis (iv), it is possible to find $c\in\R$
such that
$$
   0<c<\min\left\{1,c[F(\bar p,\cdot)](\bar x,\nullv) \right\}.
$$
By definition of $c[F(\bar p,\cdot)](\bar x,\nullv)$, corresponding
to $c$ one can find $\epsilon_c>0$ such that
\begin{eqnarray}    \label{in:c0Fpositive}
    \inf\{  \|x^*\|:\ x^*\in\Coder F(\bar p,\cdot)(x,y)(y^*),\ 
       x\in\ball(\bar x,\epsilon_c),\ y\in\ball(\nullv,\epsilon_c)
               \backslash\{\nullv\},\ 
        (x,y)\in\grph F(\bar p,\cdot),\  \|y^*\|=1\}>c.
\end{eqnarray}
Now fix $(x,y)\in\ball(\bar x,\epsilon_c/2)\times\ball(\nullv,
\epsilon_c/2)$, with $0<\|y\|+\iota((x,y),\grph F(\bar p,\cdot))
\le\epsilon_c/2$. Notice that from the last inequality one can
deduce that $(x,y)\in\grph F(\bar p,\cdot)$. Moreover, being
$$
    \dispnull (x,y)=\|y\|+\iota((x,y),\grph F(\bar p,\cdot)),
$$
function $\dispnull :\X\times\Y\longrightarrow\R\cup\{\pm\infty\}$
is expressible as a sum of a Lipschitz continuous term and
an addend, which is l.s.c. in a neighbourhood of $(\bar x,\nullv)$
as indicator of the locally closed set $\grph F(\bar p,\cdot)$
(recall hypothesis (ii)).
Thus, it is possible to employ the fuzzy sum rule for estimating
the Fr\'echet subdifferential of $\dispnull$ at $(x,y)$.
Taken $\eta$ in such a way that
$$
   0<\eta<\min\left\{\frac{(1-c)^2c}{2},\frac{\epsilon_c}{2},
   \|y\|\right\},
$$
according to Lemma \ref{lem:fuzsumrule} for every $(x^*,y^*)\in
\fsubdif\dispnull (x,y)$ there exist:
$$
    (x_1,y_1),\, (x_2,y_2)\in\X\times\Y,\quad\hbox{with }
   x_i\in\ball(x,\eta),\ y_i\in\ball (y,\eta),\ i=1,\, 2,
$$
$$
    (x_1^*,y_1^*),\, (x_2^*,y_2^*)\in\X^*\times\Y^*,\hbox{ with }
    x_1^*=\nullv^*,\ \|y_1^*\|=1,\ \langle y_1^*,y_1\rangle=\|y_1\|,
    \quad\hbox{and } (x_2^*,-y_2^*)\in\Normal((x_2,y_2),
     \grph F(\bar p,\cdot)),
$$
such that
\begin{eqnarray}     \label{in:starinball}
    \|x^*-x_2^*\|\le\eta,\qquad \|y^*-y_1^*-y_2^*\|\le\eta,
\end{eqnarray}
and
\begin{eqnarray}     \label{in:strict0eta}
    0<| \|y\|-\|y_i\|+\iota((x_i,y_i),\grph F(\bar p,\cdot))|\le\eta,
    \quad i=1,\, 2.
\end{eqnarray}
Observe that by the last inequality it must be $(x_i,y_i)\in\grph
F(\bar p,\cdot)$, $i=1,\, 2$. Moreover, since $\eta<\epsilon_c/2$,
one has
$$
    d(x_i,\bar x)\le d(x_i,x)+d(x,\bar x)\le\frac{\epsilon_c}{2}+
   \frac{\epsilon_c}{2}=\epsilon_c,\quad i=1,\, 2,
$$
$$
    d(y_i,\nullv)\le d(y_i,y)+d(y,\nullv)\le\frac{\epsilon_c}{2}+
   \frac{\epsilon_c}{2}=\epsilon_c, \quad i=1,\, 2.
$$
Let us estimate the norm of an arbitrary $(x^*,y^*)\in\fsubdif\dispnull (x,y)$.
If $\|y^*\|\ge c$, then, by equipping the product space $\X^*
\times\Y^*$ with the norm
$$
   \|(x^*,y^*)\|=\max\left\{\frac{\|x^*\|}{(1-c)^2},\, \|y^*\|\right\},
$$
one finds
\begin{eqnarray}     \label{in:formpositive}
    \|(x^*,y^*)\|\ge \|y^*\|\ge c.
\end{eqnarray}
Otherwise, if $\|y^*\|<c$, by taking into account the second
inequality in (\ref{in:starinball}), one obtains
\begin{eqnarray}      \label{in:y*triang}
    \|y_2^*\|\ge \|y^*-y_1^*\|-\eta \ge \|y^*\|-\|y_1^*\|-\eta=1-\|y^*\|-\eta
   >1-c-\eta>1-c-(1-c)c=(1-c)^2,
\end{eqnarray}
as, being $0<c<1$, it is $\eta<\frac{(1-c)^2c}{2}<(1-c)c$.
The last inequalities enable one to set
$$
    y_0^*=\frac{y_2^*}{\|y_2^*\|},\qquad x_0^*=\frac{x_2^*}{\|y_2^*\|},
    \qquad x_0=x_2\qquad\hbox{and}\qquad y_0=y_2.
$$
Since $(x_2^*,-y_2^*)\in\Normal((x_2,y_2),\grph F(\bar p,\cdot))$,
one has $x_2^*\in\Coder F(\bar p,\cdot)(x_2,y_2)(y_2^*)$ and hence,
by the positive homogeneity of the Fr\'echet coderivative, it
results in
$$
    x_0^*\in\Coder F(\bar p,\cdot)(x_0,y_0)(y_0^*),\quad\hbox{with }
    \|y_0^*\|=1.
$$
Besides, on account of (\ref{in:strict0eta}) and of the choice of
$\eta$, it is
$$
   \|y_0\|\ge\|y\|-\eta>0
$$
and hence $y_0\ne\nullv$. With the above position, by virtue of the
first inequality in (\ref{in:starinball}), one finds
$$
    \|(x^*,y^*)\|\ge \frac{\|x^*\|}{(1-c)^2}\ge \frac{\|x_2^*\|-\eta}
    {(1-c)^2}\ge\frac{\|y_2^*\| \|x_0^*\|-\eta}{(1-c)^2}.
$$
Therefore, since $x_0\in\ball(\bar x,\epsilon_c)$, $y_0\in\ball
(\nullv,\epsilon_c)\backslash\{\nullv\}$ and $(x_0,y_0)\in\grph F
(\bar p,\cdot)$, in the light of inequality
(\ref{in:c0Fpositive}), by recalling (\ref{in:y*triang}) one obtains
$$
    \|(x^*,y^*)\|\ge \frac{1}{(1-c)^2}\left[(1-c)^2c-\frac{(1-c)^2c}{2}\right]
     =\frac{c}{2}.
$$
The latter estimation of $\|(x^*,y^*)\|$, along with (\ref{in:formpositive}),
leads to conclude that
$$
    \sostsl\dispnull (\bar x,\nullv)\ge\frac{c}{2},
$$
thereby completing the proof.
\end{proof}

\begin{remark}
The non degeneration condition appearing in Proposition \ref{pro:calmAspcod}
(hypothesis (iv)) requires to compute Fr\'echet coderivatives of
$F(\bar p,\cdot)$ at each point of a set like $[\ball (\bar x,\epsilon)\times
(\ball (\nullv,\epsilon)\backslash\{\nullv\})]\cap\grph F(\bar p,
\cdot)$. Such a drawback typically arises in many other regularity
conditions. Nevertheless, it is worthwhile mentioning the possibility
to pass to an one-point condition, by replacing the basic Fr\'echet
coderivatives with a single limiting coderivative construction, as
indicated for instance in \cite{ChKrYa11} (where more details can be found).
\end{remark}

In the case in which a nonnull but smooth base term enters
$(\mathcal{GE}_p)$, by strenghtening the assumption on the space
$(\Y,\|\cdot\|)$, a further calmness criterion may be formulated
in the following way.

\begin{proposition}
Let $f:P\times\X\longrightarrow\Y$ be a function and let $F:P\times\X
\longrightarrow 2^\Y$ be a set-valued mapping defining a problem
$(\mathcal{GE}_p)$, with solution mapping $G:P\longrightarrow 2^\X$,
and let $\bar x\in G(\bar p)$ and $\bar y=f(\bar p,\bar x)$. Suppose
that:

(i) $(\X,\|\cdot\|)$ is Asplund and $(\Y,\|\cdot\|)$ is a Fr\'echet
smooth renormable Banach space;

(ii) the graph of $F(\bar p,\cdot):\X\longrightarrow 2^\Y$ is closed
in a neighbourhood of $(\bar x,\bar y)$;

(iii) $f$ is calm with respect to $p$ at $\bar p$, uniformly in $x$ in a
neighbourhood of $\bar x$;

(iv) function $f(\bar p,\cdot)$ is Fr\'echet differentiable at
$\bar x$ and continuous in a neighbourhood of the same point;

(v) there exist $\delta>0$ and $\ell>0$ such that
$$
    F(p,x)\subseteq\ball(F(\bar p,x),\ell d(p,\bar p)),\quad
    \forall x\in\ball(\bar x,\delta),\quad\forall p\in\ball
    (\bar p,\delta);
$$

(vi) there exist positive $\gamma$ and $\epsilon_\gamma$ such
that
$$
     \forall (x,y)\in\ball (\bar x,\epsilon_\gamma)\times \ball (\bar x,
     \epsilon_\gamma), \hbox{ with } (x,y)\in\grph F(\bar p,\cdot),\ 
     f(\bar p,x)\ne y,\hbox{ and } \|f(\bar p,x)-y\|\le\epsilon_\gamma
$$
it  holds
\begin{eqnarray}   \label{in:regcondFrdiff}
    \inf_{\|y^*\|=1} \|\FrDer f(\bar p,\cdot)(x)^*y^*\|>(1+\gamma)
    \|\Coder F(\bar p,\cdot)(x,y)\|_+ +\gamma.
\end{eqnarray}
Then, $G$ is calm at $(\bar p,\bar x)$.
\end{proposition}

\begin{proof}
Recall that, according to the Ekeland-Lebourg theorem, any
Banach space admitting an equivalent Fr\'echet differentiable norm
is an Asplund space (see \cite{EkeLeb76}). Therefore, both
$(\X,\|\cdot\|)$ and $(\Y,\|\cdot\|)$ are Asplund spaces. It is then
possible to exploit once again the subdifferential representation of the
strict outer slope mentioned in Remark \ref{rem:stsllipder}(ii),
according to which one has
\begin{eqnarray*}
    \sostsl{\disp}(\bar x,\bar y) &=& \soesubdifsl{\disp}(\bar x,\bar y) \\
     &= & \liminf_{\epsilon\to 0^+} \{ \|(x^*,y^*)\|: (x^*,y^*)\in\fsubdif
              \disp (x,y),  \\
      & &  (x,y)\in \ball (\bar x,\epsilon)\times \ball (\bar x,\epsilon),
                0<\disp(x,y)\le\epsilon\}.
\end{eqnarray*}
Fix an arbitrary $(x,y)\in \ball(\bar x,\epsilon_\gamma)\times
\ball (\bar x,\epsilon_\gamma)$, with $0<\disp(x,y)\le\epsilon_\gamma$.
The finiteness of the value $\disp(x,y)$ implies that $(x,y)\in\grph F
(\bar p,\cdot)$, whereas its positivity entails that $f(\bar p,x)
\ne y$. Since the norm of $\Y$ can be assumed to be Fr\'echet
differentiable at $f(\bar p,x)-y$ by hypothesis (i), function
$(u,v)\mapsto \|f(\bar p,u)-v\|$ is Fr\'echet differentiable at
$(x,y)$. According to known Fr\'echet subdifferential calculus rules,
it holds
$$
   \fsubdif\disp (x,y)=\FrDer\|f(\bar p,\cdot)-\cdot\|(x,y)+
    \Normal((x,y),\grph F(\bar p,\cdot)).   
$$
Thus, if $(x^*,y^*)\in \fsubdif\disp (x,y)$, there exist
$$
    v^*_1\in\Y^*,\hbox{ with } \|v^*_1\|=1 \hbox{ and }
    \langle v^*_1,f(\bar p,x)-y\rangle=\|f(\bar p,x)-y\|,
$$
and
$$
    (u^*_2,v^*_2)\in\X^*\times\Y^*, \hbox{ with } u^*_2\in
    \Coder F(\bar p,\cdot)(x,y)(v^*_2),
$$
such that
$$
    (x^*,y^*)=(\FrDer f(\bar p,\cdot)(x)^*v^*_1+u^*_2,-v^*_1-v^*_2).
$$
One has now to verify that the norm of $(x^*,y^*)$ remains
greater than a positive constant. Similarly as for the proof of the previous
proposition, $\X^*\times\Y^*$ is assumed to be equipped with the
norm $\|(x^*,y^*)\|=\max\{\|x^*\|,\, \|y^*\|\}$. In the case $v^*_2=\nullv^*$,
one readily finds
$$
    \|(x^*,y^*)\|\ge\|v^*_1\|=1.
$$
In the case $v^*_2\ne\nullv^*$, let us distinguish two cases.
If it is $\|v^*_2\|>1+\gamma$, one obtains at once
$$
    \|(x^*,y^*)\|\ge\|v^*_2+v^*_1\|\ge\|v^*_2\|-\|v^*_1\|>\gamma.
$$
Otherwise, if it is $\|v^*_2\|\le 1+\gamma$, set
$$
    v_0^*=\frac{v^*_2}{\|v^*_2\|}\qquad\hbox{and}\qquad 
    u_0^*=\frac{u^*_2}{\|v^*_2\|},
$$
so that $u_0^*\in\Coder F(\bar p,\cdot)(x,y)(v_0^*)$.
Since it holds
$$
    \|u^*_2\|=\|u_0^*\|\|v^*_2\|\le\|\Coder F(\bar p,\cdot)(x,y)\|_+
    \|v^*_2\|\le(1+\gamma)\|\Coder F(\bar p,\cdot)(x,y)\|_+,
$$
in force of condition (\ref{in:regcondFrdiff}), it results in
\begin{eqnarray*}
    \|(x^*,y^*)\|& \ge& \|\FrDer f(\bar p,\cdot)(x)^*v^*_1+u^*_2\|\ge
    \|\FrDer f(\bar p,\cdot)(x)^*v^*_1\|-\|u^*_2\|  \\
    & \ge &
    \inf_{\|v^*\|=1}\|\FrDer f(\bar p,\cdot)(x)^*v^*\|-(1+\gamma)
    \|\Coder F(\bar p,\cdot)(x,y)\|_+  > \gamma.
\end{eqnarray*}
Thus, in any case, one can conclude that
\begin{eqnarray*}
    \inf \{ \|(x^*,y^*)\|: & & (x^*,y^*)\in\fsubdif\disp (x,y),\  (x,y)
     \in \ball (\bar x,\epsilon_\gamma)\times \ball (\bar x,
     \epsilon_\gamma), \\
     & & 0<\disp(x,y)\le\epsilon_\gamma\} >\min\{1,\gamma\}.
\end{eqnarray*}
By virtue of the criterion expressed by Theorem \ref{thm:impcalm},
what obtained allows to get the thesis.
\end{proof}


\section{Applications to parametric constrained optimization}     \label{sec:applparoptim}

The stability analysis of parametric optimization problems is
a well developed and active field of research, dealing with
qualitative and quantitative investigations about the behaviour
of the optimal solution set and of the optimal value function in the
presence of perturbations.
One of the early monograph to be mentioned, entirely devoted to a
systematic presentation of this topic, is \cite{Fiac83}.
It was followed by a good amount of works (see for example
\cite{BonSha00,KlaKum02,Yen97} and the references therein).

Some results concerning the value function associated with
a familly of parametric constrained optimization problems
have been already exposed in Section \ref{sec:varanapre}.
In the current section the analysis is focussed on the quantitative
stability behaviour of the solution set mapping. Indeed, 
the first result presented provides a sufficient
condition for the Lipschitz lower semicontinuity of such
set-valued mapping, which works in a purely metric space setting.

\begin{proposition}      \label{pro:Liplscpco}
With reference to a family $(\mathcal{P}_p)$ of perturbed
problems, let $\bar p\in P$ and $\bar x\in\Argmin(\bar p)$. Suppose
that:

(i) $(X,d)$ is complete;

(ii) for every $p\in P$ near $\bar p$, function $x\mapsto\varphi(p,x)$
is continuous in a neighbourhood of $\bar x$;

(iii) function $h$ is locally Lipschitz around $(\bar p,\bar x)$ with
respect to $x$, uniformly in $p$, i.e. there exist positive $\kappa$
and $\delta_\kappa$ such that
$$
   d(h(p,x_1),h(p,x_2))\le\kappa d(x_1,x_2),\quad\forall
   p\in B(\bar p,\delta_\kappa),\ \forall x_1,\, x_2\in B(\bar x,
   \delta_\kappa);
$$

(iv) functions $p\mapsto\varphi(p,\bar x)$, $p\mapsto h(p,\bar x)$
and $\valf$ are calm at $\bar p$;

(v) it holds
$$
    \sostslx{\varphi}(\bar p,\bar x)>\kappa.    
$$

\noindent Then $\bar p\in\inte(\dom\Argmin)$ and $\Argmin$ is Lipschitz
l.s.c. at $(\bar p,\bar x)$.
\end{proposition}

\begin{proof}
The first part of the thesis is obviously a consequence of the second one.
The latter is readily proved by applying Theorem \ref{thm:impLiplsc}
to the generalized equation $(\mathcal{GE}_p)$ defined by
a base $f:P\times X\longrightarrow (\R\cup\{\pm\infty\})\times Y$
given by
$$
    f(p,x)=(\varphi(p,x)-\valf(p),h(p,x))
$$
and by the constant field $F:P\times X\longrightarrow 2^{\R\cup
\{\pm\infty\}}\times Y$
$$
    F(p,x)=\{0\}\times C.
$$
In fact, the associated variational system coincides with the
set-valued mapping $\Argmin:P\longrightarrow 2^X$. Thus,
one has to verify that, in the case under consideration,  all
hypotheses of Theorem \ref{thm:impLiplsc} are fulfilled.
Since $F$ is a constant set-valued mapping and
$$
    f(\bar p,\bar x)=(0,h(\bar p,\bar x))\in F(p,\bar x)=\{0\}\times C,
    \quad\forall p\in P,
$$
hypotheses (ii) and (iii) are evidently satisfied. Next, hypothesis
(iv) is in force because of the continuity of functions $x\mapsto
\varphi(p,x)$ and of (local Lipschitz) continuity of functions
$x\mapsto h(p,x)$, around $\bar x$, with $p$ in a neighbourhood
of $\bar p$.
Again, calmness of
$f(\cdot,\bar x)$ at $\bar p$ (hypothesis (v)) follows from calmness
of $\varphi(\cdot,\bar x)$, $\valf$, and $h(\cdot,\bar x)$ at the
same point. It remains to show that also hypothesis (vi) is
valid. To this aim, observe that, if equipping the product space
$(\R\cup\{\pm\infty\})\times Y$ with the sum metric, one
obtains
\begin{eqnarray}      \label{eq:pdispappl}
    \pdisp(p,x)=\dist((\varphi(p,x)-\valf(p),h(p,x)),\{0\}\times C)
    =\varphi(p,x)-\valf(p)+\dist(h(p,x),C).
\end{eqnarray}
Since function $x\mapsto h(p,x)$ is locally Lipschitz around
$(\bar p,\bar x)$ with constant $\kappa>0$ and function $y\mapsto
\dist(y,C)$ is Lipschitz continuous with constant $1$ all over $Y$,
their composition $x\mapsto \dist (h(p,x),C)$ turns out to be
locally Lipschitz with constant $\kappa$ around $\bar x$,
uniformly in $p$ near $\bar p$.
By recalling what has been observed in Remark \ref{rem:stsllipder},
being $\sostslx{\varphi}(\bar p,\bar x)>\kappa$, one obtains from
(\ref{eq:pdispappl})
$$
    \sostslx{\pdisp}(\bar p,\bar x)\ge\sostslx{\varphi}(\bar p,\bar x)
   -\kappa>0.
$$
This completes the proof.
\end{proof}

Relying on the nice properties enjoyed by the Fr\'echet subdifferential
calculus in Asplund spaces, a stability result for perturbed nonsmooth
optimization problems can be formulated as follows.

\begin{corollary}
With reference to a family $(\mathcal{P}_p)$ of perturbed
problems, let $\bar p\in P$ and $\bar x\in\Argmin(\bar p)$. Suppose
that $(\X,\|\cdot\|)$ is Asplund and that
hypotheses (ii), (iii) and (iv) of Proposition \ref{pro:Liplscpco}
are in force. If
$$
    \lim_{\epsilon\to 0^+}\ \inf\{ \|x^*\|:\ 
    (p,x)\in\ball(\bar p,\epsilon)\times\ball(\bar x,\epsilon),\ 
     \varphi(\bar p,\bar x)<\varphi(p,x)\le\varphi(\bar p,\bar x)+
    \epsilon,\ x^*\in\widehat\partial_x\varphi(p,x)\}>\kappa,
$$
where $\kappa$ is the Lipschitz constant as in hypothesis (iii),
then $\bar p\in\inte(\dom\Argmin)$ and $\Argmin$ is Lipschitz
l.s.c. at $(\bar p,\bar x)$.
\end{corollary}

\begin{proof}
The thesis follows immediately after having estimated $\sostslx{\varphi}
(\bar p,\bar x)$ in terms of the partial Fr\'echet subdifferentials
of $\varphi$ near $(\bar p,\bar x)$, as indicated in Remark
\ref{rem:stsllipder}(ii).
\end{proof}

On  the base of simple representations of the strict outer slope
enabled by the presence of differentiability, for optimization
problems with smooth data it is possible to reformulate
the above criterion in general Banach spaces.

\begin{corollary}
With reference to a family $(\mathcal{P}_p)$ of perturbed
problems, let $\bar p\in P$ and $\bar x\in\Argmin(\bar p)$. 
Suppose that:

(i) $(\X,\|\cdot\|)$ is a Banach space;

(ii) for every $p$ near $\bar p$, each function $x\mapsto\varphi(p,x)$
is strictly differentiable at $\bar x$, with derivative
$\StDer_x\varphi(p,\bar x)$; 

(iii) there exists $\delta>0$ such that each function $x\mapsto h
(p,x)$ is G\^ateaux differentiable on $\ball(\bar x,\delta)$, for
every $p$ near $\bar p$, with derivative $\GDer_xh(p,x)$, and it holds
$$
    \|\StDer_x\varphi(\bar p,\bar x)\|>\sup_{(p,x)\in\ball(\bar p,\delta)
    \times \ball(\bar x,\delta)} \|\GDer_xh(p,x)\|_{\mathcal{L}};
$$

(iv) functions $p\mapsto\varphi(p,\bar x)$, $p\mapsto h(p,\bar x)$
and $\valf$ are calm at $\bar p$.

\noindent Then $\bar p\in\inte(\dom\Argmin)$ and $\Argmin$ is Lipschitz
l.s.c. at $(\bar p,\bar x)$.
\end{corollary}

\begin{proof}
Since function $h(p,\cdot)$ is G\^ateaux differentiable on the convex
subset $\ball(\bar x,\delta)$, for each $p$
near $\bar p$, then as a consequence of the mean value theorem one has
$$
    \|h(p,x_1)-h(p,x_2)\|\le\sup_ {x\in \ball(\bar x,\delta)}
    \|\GDer_xh(p,x)\|_\mathcal{L}\|x_1-x_2\|,\quad\forall
    x_1,\, x_2\in \ball(\bar x,\delta).
$$
In other words, in the light of hypothesis (iii) $h$ is locally
Lipschitz with respect to $x$, uniformly in $p$, around $(\bar p,
\bar x)$, having a Lipschitz constant not greater than $\displaystyle
\sup_{(p,x)\in\ball(\bar p,\delta)\times \ball(\bar x,\delta)} \|\GDer_x
h(p,x)\|_{\mathcal{L}}$. By hypothesis $(ii)$, function $\varphi
(p,\cdot)$ is continuous in a neighbourhood of $\bar x$, for every
$p$ near $\bar p$. Moreover, as observed in Remark \ref{rem:stsllipder}(iii),
it holds $\sostslx{\varphi}(\bar p,\bar x)=\|\StDer_x\varphi(\bar p,\bar x)\|$.
Thus, being all its hypotheses fulfilled, it
remains to apply Proposition \ref{pro:Liplscpco}.
\end{proof}


\vskip.5cm

\vskip1cm   

\bibliographystyle{amsplain}

\end{document}